\documentclass[12pt]{article}
\pdfoutput=1
\usepackage{fullpage}

\usepackage[utf8]{inputenc}

\usepackage{graphicx}
\usepackage{amsmath}
\usepackage{amssymb}

\usepackage{caption}
\usepackage{hyperref}
\usepackage{eurosym}
\usepackage{dirtytalk}
\let\vec\relax
\DeclareMathAccent{\vec}{\mathord}{letters}{"7E}
\usepackage{accents}
\usepackage{fancyvrb}
\usepackage{appendix}

\ifdefined\theorem\else \newtheorem{theorem}{Theorem}[section]\fi
\ifdefined\proposition\else \newtheorem{proposition}[theorem]{Proposition}\fi
\ifdefined\definition\else \newtheorem{definition}[theorem]{Definition}\fi
\ifdefined\lemma\else \newtheorem{lemma}[theorem]{Lemma}\fi
\ifdefined\corollary\else \newtheorem{corollary}[theorem]{Corollary}\fi
\ifdefined\remark\else \newtheorem{remark}[theorem]{Remark}\fi
\ifdefined\assumption\else \newtheorem{assumption}{Assumption}\fi
\ifdefined\example\else \newtheorem{example}{Example}\fi

\usepackage{stmaryrd} 
\usepackage[colors]{optsys}

\usepackage[square,numbers]{natbib}
\usepackage{subcaption}
\usepackage{multirow}
\usepackage{schemabloc}
\usepackage{mathrsfs}
\usepackage{algorithm}
\usepackage[algo2e, linesnumbered]{algorithm2e}

\newcommand{\pRR}{\RR_+}
\newcommand{\spRR}{\RR_{+}^*}
\newcommand{\Borel}{\mathcal{B}}

\newcommand{\StageCost}{L}
\newcommand{\Final}{K}
\newcommand{\ConstraintSet}{\mathcal{U}}
\newcommand{\SolutionSet}{\mathcal{U}\opt}
\renewcommand{\ValueFunction}{V}
\newcommand{\QFunction}{Q}
\newcommand{\indicator}{\iota}

\newcommand{\param}{p}

\newcommand{\Param}{\mathcal{P}}

\newcommand{\PARAM}{\mathbb{P}}
\newcommand{\ParamAd}{\mathcal{P}_{\text{ad}}}

\newcommand{\nstate}{n_\state}
\newcommand{\ncontrol}{n_\control}
\newcommand{\nuncertain}{n_\uncertain}
\newcommand{\nparam}{n_\param}

\newcommand{\prof}{p}
\newcommand{\del}{r} 
\newcommand{\gen}{g}
\newcommand{\soc}{s}

\newcommand{\rvControl}{\va{\Control}}
\newcommand{\rvState}{\va{\State}}
\newcommand{\rvUncertain}{\va{\Uncertain}}

\newcommand{\minBat}{\underline{u}}
\newcommand{\maxBat}{\overline{u}}
\newcommand{\battery}{(\kappa, \minBat, \maxBat, \rho_c, \rho_d)}
\newcommand{\peak}{\overline{q}}
\newcommand{\Penalty}{L^\text{p}}
\newcommand{\Energy}{L^\text{e}}

\DeclareMathAccent{\wtilde}{\mathord}{largesymbols}{"65}
\newcommand{\SmoothValueFunction}{\underaccent{\wtilde}{V}^\mu}
\newcommand{\RawSmoothValueFunction}{\underaccent{\wtilde}{V}}
\newcommand{\RawSmoothQFunction}{\underaccent{\wtilde}{Q}}
\newcommand{\statebis}{\tilde{\state}}

\newcommand{\PolyhedralValueFunctionbis}{\underline{\ValueFunction}}
\newcommand{\maxiter}{k}

\renewcommand{\PARAM}{\RR^{\nparam}} 
\newcommand{\PARAMslice}{\RR^{\nparam}}
\renewcommand{\STATE}{\RR^{n_x}}
\renewcommand{\CONTROL}{\RR^{n_u}}
\renewcommand{\UNCERTAIN}{\RR^{n_w}}

\newcommand{\Val}{\ ]{-}\infty, +\infty]}

\newcommand{\Pairxu}{Y}

\newcommand{\sep}{;}

\usepackage{a4wide}
\usepackage{authblk}

\usepackage{fullpage}
\graphicspath{{Figures/}}

\def\keywords#1{\par\addvspace\medskipamount{\rightskip=0pt plus1cm
\def\and{\ifhmode\unskip\nobreak\fi\ $\cdot$
}\noindent\keywordname\enspace\ignorespaces#1\par}}
\def\keywordname{{\bfseries Keywords}}%

\title{Differentiability and Regularization \\
  of Parametric Convex Value Functions \\
  in Stochastic Multistage Optimization}

\author[1]{Adrien~Le~Franc}
\author[2]{Pierre~Carpentier}
\author[3]{Jean-Philippe~Chancelier}
\author[3]{Michel De~Lara}

\affil[1]{LAAS CNRS, Toulouse, France}
\affil[2]{UMA, ENSTA Paris, IP Paris, Palaiseau, France}
\affil[3]{CERMICS, ENPC, Institut Polytechnique de Paris, CNRS, Marne-la-Vallée, France}

\begin{document}
\maketitle

\begin{abstract}
  In multistage decision problems, it is often the case that an initial
  strategic decision (such as investment) is followed by a sequence of operational ones
  (operating the investment).
  Such initial strategic decision can be seen as a parameter affecting a
  multistage decision problem.
  More generally, we study in this paper a standard multistage stochastic
  optimization problem depending on a parameter chosen at the initial stage.
  When the parameter is fixed, Stochastic Dynamic Programming
  provides a way to compute the optimal value of the problem.
  Thus, the value function depends both on the state (as usual) and on the parameter.
  Our aim is to  investigate on the possibility to efficiently compute
  gradients of the value function with respect
  to the parameter, when these objects exist.
  When nondifferentiable, we propose a regularization method based on
  the Moreau-Yosida envelope. We present a numerical test case
  from day-ahead power scheduling.
\end{abstract}


\section{Introduction}

We consider optimization problems where an \emph{upstream} decision
is made in the first place, which stands for a parameter for a
\emph{downstream} multistage stochastic optimization problem.
Our work is motivated by applications in the field of energy
planning, where such decision structure arises naturally.
As a typical example, the regulatory rules considered
in~\cite{team2015impact, n2019optimal, pflaum2017battery}
impose renewable power plants
to commit a day-ahead power production profile,
upstream to the intraday management phase
where costs are subject to uncertainties
arising from power production.
Another important application arises when dealing with
large-scale stochastic multistage optimization problems. In many
cases, dualizing some coupling constraint allows for decomposition
into subproblems, each corresponding to a ``small'' stochastic
multistage optimization problem. The  Lagrange multiplier associated
with the coupling constraint has then to be optimized, and can be interpreted
as a parameter for the multistage stochastic optimization subproblems
(see \cite{brown2022strength,carpentier2020mixed} for details).

In this article, we propose a standard formulation for
\emph{parametric multistage stochastic optimization problems} (PMSOP).
In the formulation we outline, the parameter does not affect the dynamics
but affects all instantaneous (and final) costs at all periods,
and also possibly the admissible control sets;
due to this structure, the parameter cannot be identified with the initial
decision of a multistage stochastic optimization problem.
When the value of the parameter is fixed, Stochastic Dynamic Programming
(see e.g. \cite{bertsekas1995dynamic, puterman94}) is a way to obtain
the value of the downstream problem by computing the value functions
given by the Bellman equation. Thus, the value functions now depend
both on the state (as usual) and on the parameter.
On top of that, we investigate on the possibility to efficiently compute
additional first-order information, e.g. gradients of the value functions
with respect to the parameter, when these objects exist.
Our end goal is to formulate first-order oracles
which let us enter the world of (primal)
first-order optimization methods (see \cite{beck2017first}
for a recent survey) to solve PMSOPs, that is, to perform
optimization with respect to the parameter.

Of course, the interest in such kinds of problems is not new.
The reference textbook of Bonnans and Shapiro \cite{bonnans2013perturbation}
gathers numerous results
on the value functions of a parameterized optimization problem.
In the context of multistage stochastic programming,
the sensitivity analysis of the value of a downstream problem
with respect to some model parameters has been already studied
in~\cite{cen:inria-00579668, guigues2023duality, tercca2020envelope}.
These works mainly
focus on the computation of directional derivatives of the value
function, in the case where the stage cost functions of the problem
are affine. In~\cite{cen:inria-00579668}, the authors further argue
that --- since Danskin's Theorem tells us that the value function
is locally Lipschitz continuous --- the value function is differentiable
almost everywhere by Rademacher's Theorem.
A similar conclusion is drawn in~\cite{tercca2020envelope},
and formulas to compute the gradient at points where the value
function is differentiable are given in both references.
However, employing smooth optimization methods to
minimize nondifferentiable functions can yield suboptimal solutions,
even in the convex case, as illustrated by the example discussed in \cite[\S8.1.2]{beck2017first}.

Thus, our work differs from the above references in at least two points: $(i)$
we study the existence and provide formulas for the gradient of a parametric
value function, whereas previous works concentrate on directional derivatives;
$(ii)$ we consider convex nonlinear stage costs and constraints, going beyond
the usual linear multistage stochastic programming framework.  Also, to cover
cases where the parametric value functions are convex with respect to their
parameter argument but nondifferentiable, we propose a regularization method
based on the Moreau-Yosida envelope \cite{moreau1965proximite,
  yosida1971functional}.  We study the convergence properties of both the
regularized value functions that we introduce, and of the parameter solutions of
a PMSOP, as our regularization coefficient tends to zero.  Although
Moreau-Yosida regularization has been previously employed in optimal control
problems (see e.g. \cite{bergounioux1998comparison}), few anterior studies
concentrate on stochastic problems, except recently in~\cite{ortega2021moreau},
where the authors examine the case of a discounted infinite horizon Markov
decision process.  We share common interests with the latter reference, but our
finite horizon context leads us to follow a different path. Finally, we propose
an alternative method (based on SDDP with an extended state) and an assessment
technique (based on SDDP with the original state) to evaluate the quality of a
parameter as a solution to a PMSOP. For both purposes, we rely on the stochastic
dual dynamic programming algorithm (SDDP)
\cite{pereira1991multi,shapiro2011analysis,%
  philpott2008convergence, girardeau2015convergence}.

The paper is organized as follows.
First, in Sect.~\ref{sec:parametric_multistage_problem},
we introduce the definition of a PMSOP
and of parametric value functions.
Second, in Sect.~\ref{sec:differentiable_convex_value_functions},
we provide conditions to obtain Bellman-like equations for
the gradient of differentiable convex value functions
with respect to their parameter argument.
To extend the method to convex nondifferentiable parametric value
functions, we propose a regularization scheme and we study
the convergence properties of resulting regularized parametric
value functions.
From these theoretical results, we deduce,
in Sect.~\ref{sec:experimental_assessment_based_on_SDDP},
a first-order optimization method based on recursive gradient computation
to solve PMSOPs.
Finally, in~Sect.~\ref{sec:day_ahead_problem},
we present a numerical test case inspired
from day-ahead power scheduling. 

\subsubsection*{Background Notions and Notations}


\textit{Natural numbers.}
We use the notation \(
\ic{i,j}=\na{i,i+1,\ldots,j-1,j}
\) for any pair of natural numbers such that \( i \leq j \).

\noindent
\textit{Probability.}
Let $(\Omega, \mathcal{F}, \PP)$ be a probability space.
We use bold capital letters, e.g. $\va{Z}$,
to denote random variables,
and denote by $\sigma\np{\va{Z}}$ the $\sigma$-algebra
on~$\Omega$ ($\sigma\np{\va{Z}} \subset \mathcal{F}$)
generated by the random variable~$\va{Z}$
and,  when the random variable~$\va{Z}$ takes a finite number of values,
by $\Support{\va{Z}}$ the support of~$\va{Z}$, that is,
the set of possible values of~$\va{Z}$ with positive probabilty.
Besides, for a topological space $\XX$,
we denote by $\Borel(\XX)$ its Borel $\sigma$-field.

\noindent
\textit{Functional analysis and topology.}
We introduce the extended real line $\barRR = [-\infty, +\infty]$,
and we denote $\pRR = [0, +\infty[$ and $\spRR = \ ]0, +\infty[$.
Let $f : \XX \to \barRR$ be a function.
The effective domain $\dom f$
is the set $\defset{x \in \XX}{f(x)<+\infty}$,
and the function~$f$ is said to be proper
if $f > -\infty$ and $\dom f \neq \emptyset$.
For any subset \( X \subseteq \XX \),
$\indicator_{X} : \XX \to \barRR $ denotes the indicator
function of the set~$X$:
\( \indicator_{X}\np{x} = 0 \) if \( x \in X \),
and \( \indicator_{X}\np{x} = +\infty \)
if \( x \not\in X \).
For a topological space $\XX$ and \( X \subseteq \XX \),
we recall that the interior $\interior X$ is defined as
the largest open set contained in $X$.

\section{Parametric Multistage Stochastic Optimization Problems}
\label{sec:parametric_multistage_problem}

In~\S\ref{sec:problem_formulation}, we introduce a standard formulation
for PMSOPs.  Then, in~\S\ref{sec:parametric_value_functions},
we introduce parametric value functions defined by the Bellman
equations.
Finally, we discuss the specific role of the parameter in~\S\ref{sec:role_of_parameter}.

\subsection{Problem Formulation}
\label{sec:problem_formulation}

We are interested in solving problems of the form
\begin{subequations}
  \label{eq:parametric_multistage_problem}
  \begin{equation}
    \Min_{\param \in \ParamAd} \Phi(\param) \eqfinv
    \label{eq:upstream_problem}
  \end{equation}
  in the case where the objective function
  $\Phi$ in~\eqref{eq:upstream_problem}
  is the value of the following \emph{parametric multistage stochastic
    optimization problem} (PMSOP):
  \begin{align}
    \Phi(\param) =
    & \inf_{\rvControl_0, \ldots, \rvControl_{\horizon-1}}
      \Besp{\sum_{t=0}^{\horizon-1} \ \coutint_t(\rvState_t, \rvControl_t, \rvUncertain_{t+1}, \param)
      + \Final(\rvState_\horizon, \param)} \eqfinv
      \label{eq:multistage_criterion} \\[0.2cm]
    & \rvState_{0} = \state_{0} \eqfinv
      \label{eq:constraints_begin}
    \\
    & \rvState_{t+1} = \dynamics_t(\rvState_{t}, \rvControl_t,
      \rvUncertain_{t+1})
      \eqsepv \forall t \in \ic{0, \horizon-1} \eqfinv
      \label{eq:state_transition}
    \\
    & \sigma(\rvControl_t) \subseteq \sigma(\rvUncertain_1, \ldots, \rvUncertain_t) \eqsepv
      \forall t \in \ic{0, \horizon-1}
      \eqfinp \label{eq:constraints_end}
  \end{align}
\end{subequations}

We now comment on all terms in Problem~\eqref{eq:parametric_multistage_problem},
and we discuss 
assumptions ensuring that the expected value
in~\eqref{eq:multistage_criterion} is well defined.  We consider a
\emph{discrete time span}
\begin{equation}
  \ic{0, \horizon} = \na{0, 1, \ldots, \horizon-1, \horizon} \eqfinv
\end{equation}
with \emph{horizon} a natural number~$\horizon \in \NN^*$.
\begin{subequations}
Concerning the \emph{upstream problem}~\eqref{eq:upstream_problem},
the variable
\begin{equation}
  \label{eq:parameter}
  \param \in \PARAM \eqfinv
\end{equation}
where $\nparam \in \NN^*$, is a \emph{parameter} which
may be chosen in the \emph{parameter set}
\begin{equation}
\label{eq:parameter_constraint_set}
  \ParamAd \subseteq \PARAM \eqfinp
\end{equation}
The parameter~$\param$, as well as the control~$\rvControl_0$, are both
(constant) decisions made at the initial time~$t=0$.  However, the
parameter~$\param$ and the control~$\rvControl_0$ play different roles.  This
point is discussed in~\S\ref{sec:role_of_parameter}.
\end{subequations}

Concerning the \emph{downstream problem}~\eqref{eq:multistage_criterion}--\eqref{eq:constraints_end},
we introduce random variables
\begin{subequations}
  \begin{align}
    \rvState_t : \np{\Omega, \mathcal{F}, \PP}
    &\rightarrow \bp{\STATE, \Borel(\STATE)} \eqsepv \forall t \in \ic{0, \horizon} \eqfinv
      \label{eq:rvState} \\
    \rvControl_t : \np{\Omega, \mathcal{F}, \PP}
    &\rightarrow \bp{\CONTROL, \Borel(\CONTROL)} \eqsepv \forall t \in \ic{0, \horizon-1} \eqfinv
      \label{eq:rvControl} \\
    \rvUncertain_t : \np{\Omega, \mathcal{F}, \PP}
    &\rightarrow \bp{\UNCERTAIN, \Borel(\UNCERTAIN)} \eqsepv \forall t \in \ic{1, \horizon} \eqfinv \label{eq:noise_process}%
  \end{align}
\end{subequations}
which denote respectively the \emph{state}, \emph{control} and \emph{noise}
variables of Problem~\eqref{eq:multistage_criterion}--\eqref{eq:constraints_end},
taking values in real Euclidean spaces of respective finite dimensions
$\np{\nstate, \ncontrol, \nuncertain} \in {\NN^*}^3$.
The state variables are initialized by $\state_0 \in \STATE$
and evolve in~\eqref{eq:state_transition} according to the 
\emph{dynamics}
\begin{equation}
  \dynamics_t : \STATE \times \CONTROL \times \UNCERTAIN \rightarrow \STATE \eqsepv \forall t \in \ic{0,\horizon-1} \eqfinp
  \label{eq:dynamics_ch2}
\end{equation}
Note that the parameter~$\param$ does not affect
the dynamics~$\dynamics_{t}$, and that the
constraints~\eqref{eq:constraints_begin}--\eqref{eq:state_transition}
are almost sure (a.s.) constraints.
The control variables are constrained by the \emph{nonanticipativity}
constraints~\eqref{eq:constraints_end}.
Lastly, the criterion to be minimized in~\eqref{eq:multistage_criterion}
is the expected value of the sum of the
\emph{parametric stage costs}
\begin{subequations}
  \label{eq:cost_functions}
  \begin{align}
    \coutint_t : \STATE \times \CONTROL \times \UNCERTAIN \times \PARAMslice
    &\rightarrow \Val \eqsepv \forall t \in \ic{0, \horizon-1} \eqfinv
      \label{eq:parametric_stage_cost}%
      \intertext{with a 
\emph{parametric final cost}}
      \Final : \STATE \times \PARAMslice
    &\rightarrow \Val \eqfinp
      \label{eq:parametric_final_cost}
  \end{align}
Notice that, whereas~$\control_0$ only appears at time~$t=0$
  in the initial cost function~\eqref{eq:cost_functions} ---
  as \( \coutint_0\np{\state_0,\control_0,\uncertain_1,\param} \) ---
$\param$ appears in the cost functions~\eqref{eq:cost_functions}
at all times. This point is discussed in~\S\ref{sec:role_of_parameter}.
\end{subequations}

\begin{remark}
\label{rm:explicit_constraints}
The cost functions can take values in $\Val$ to offer the
possibility to implicitly encode constraints through effective domains.
In practice, explicit constraints of the form $\rvControl_t \in \ConstraintSet_t\np{\rvState_t, \param}$ a.s. --- with
set-valued mappings $\ConstraintSet_t: \STATE \times \PARAMslice
\rightrightarrows \CONTROL$, for all $t \in \ic{0, T{-}1}$ ---
can be added in Problem~\eqref{eq:parametric_multistage_problem}
(see the example in Sect.~\ref{sec:day_ahead_problem}).
\end{remark}

In what follows, we provide assumptions which ensure that
the mathematical expectation~\eqref{eq:multistage_criterion} is well defined.
For this purpose, we have considered proper cost functions --- that never take the value $-\infty$ ---
and we will consider discrete noise variables~\eqref{eq:noise_process}
--- so that the expectation~\eqref{eq:multistage_criterion} is well defined.
Indeed, when noise variables~\eqref{eq:noise_process} take a finite number of
values, then so do the control variables~\eqref{eq:rvControl}
by~\eqref{eq:constraints_end} --- because, by Doob Theorem
(see \cite[Chap.~1, p.~18]{Dellacherie-Meyer:1975}), the random variable~$\rvControl_t$ is
a function of the random variables \( \rvUncertain_1, \ldots, \rvUncertain_t \).
Then all the state variables~\eqref{eq:rvState} also take a finite number of
values by the dynamics~\eqref{eq:state_transition}.
Finally, the mathematical expectation~\eqref{eq:multistage_criterion} reduces to
a finite sum of numbers that belong to~$\Val$, hence is well defined as an
element of~$\Val$.
Questioning whether our results extend to continuous noise variables
could be the subject of a following research work.

\begin{remark}
  The formulation of the nonanticipativity constraint
  in~\eqref{eq:constraints_end} corresponds to problems which formulate
  naturally in the \emph{decision-hazard} information structure \cite[p.~8]{Carpentier-Chancelier-Cohen-DeLara:2015}.
  In particular, the first decision~$\rvControl_0$ is deterministic,
  with $\sigma(\rvControl_0) = \na{\emptyset, \Omega}$.
\end{remark}

\subsection{Parametric Value Functions}
\label{sec:parametric_value_functions}



As mentioned in~\S\ref{sec:problem_formulation},
we consider discrete random variables~\eqref{eq:noise_process}
in Problem~\eqref{eq:parametric_multistage_problem}.
Besides, we make the following (discrete) white noise assumption.

\begin{assumption}[discrete white noise]
  \label{as:discrete_white_noise}
  The sequence $\sequence{\rvUncertain_t}{t \in \ic{1,\horizon}}$ of noise variables
  in~\eqref{eq:noise_process} is stagewise independent,
  and each noise variable $\rvUncertain_t$ has a finite support.
\end{assumption}
The above assumption has a direct consequence on the solutions
of the multistage Problem~\eqref{eq:multistage_criterion}-\eqref{eq:constraints_end}.
Indeed, if we consider a fixed value
of the parameter $\param \in \PARAM$ in~\eqref{eq:parameter},
we retrieve a standard \emph{multistage stochastic optimization problem}.
Therefore,
Stochastic Dynamic Programming
gives us a method for computing the optimal solution of the multistage
Problem~\eqref{eq:multistage_criterion}-\eqref{eq:constraints_end},
and thus to evaluate $\Phi(\param)$ in~\eqref{eq:upstream_problem}.

Under finite support of the noises in Assumption~\ref{as:discrete_white_noise},
each of the random variables \( \rvUncertain_1, \ldots, \rvUncertain_\horizon \)
takes a finite number of values, so that the following
\emph{Bellman equations}
\begin{subequations}
  \begin{align}
    \ValueFunction_\horizon(\state \sep \param)
    &= \Final(\state, \param) \eqsepv
      \forall (\state, \param) \in \STATE \times \PARAM \eqfinv
    \\
    \ValueFunction_t(\state \sep \param)
    &= \inf_{\control \in \CONTROL}\EE
      \Bc{\coutint_t(\state, \control, \rvUncertain_{t+1}, \param) + \ValueFunction_{t+1}
      \bp{\dynamics_t\np{\state, \control, \rvUncertain_{t+1}} \sep \param}}
      \eqsepv \label{eq:marginal_parametric_value_functions}\\
    &\hspace{3cm}
      \forall (\state, \param) \in \STATE \times \PARAM \eqsepv \forall t \in \ic{0, \horizon-1}
\notag
  \end{align}
  \label{eq:parametric_value_functions}%
\end{subequations}
are well defined and yield, by backward induction,
the sequence~$\sequence{\ValueFunction_t}{t \in \ic{0, \horizon}}$
of \emph{parametric value functions}.
%
Under stagewise independence of the noises in
Assumption~\ref{as:discrete_white_noise}, the value functions
$\sequence{\ValueFunction_t}{t \in \ic{0, \horizon}}$
in~\eqref{eq:parametric_value_functions} give the optimal value of the
multistage Problem~\eqref{eq:multistage_criterion}-\eqref{eq:constraints_end},
in the sense that
\begin{equation}
  \Phi(\param) = \ValueFunction_0(\state_{0} \sep \param) \eqsepv
  \forall \param \in \PARAM \eqfinp
  \label{eq:Phi_value_function}
\end{equation}
We refer to Bertsekas~\cite{bertsekas1995dynamic} and Puterman~\cite{puterman94}
for a comprehensive presentation of the Stochastic Dynamic Programming method.
\begin{remark}
  Although the parametric value functions
  $\sequence{\ValueFunction_t}{t \in \ic{0, \horizon}}$
  in~\eqref{eq:parametric_value_functions} take the parameter $\param$ as an
  argument, we use a semicolon ``;'' to isolate it from the state variable
  $\state$.  This notation is used to emphasize on the special role of the
  parameter in the PMSOP framework, as discussed
  in~\S\ref{sec:role_of_parameter}.
\end{remark}

Since parametric value functions are defined as
the infimum of a certain criterion in the Bellman equations,
they are \say{marginal functions},
a class of functions with rich properties \cite{bonnans2013perturbation}.
To ease applications
to our context, we introduce the (parametric) \emph{$Q$-functions}
\begin{subequations}
	\label{eq:Q_function}
	\begin{align}
	\QFunction_t(\state, \control \sep \param)
	&=
	\EE
	\Bc{\coutint_t(\state, \control, \rvUncertain_{t+1}, \param) + \ValueFunction_{t+1}
		\bp{\dynamics_t\np{\state, \control, \rvUncertain_{t+1}} \sep \param}}
	\eqfinv \label{eq:Q_function_expression} \\
	& \hspace{1cm}
	\forall (\state, \control, \param) \in
	\STATE \times \CONTROL \times \PARAM \eqsepv
	\forall t \in \ic{0, \horizon-1}
	\eqfinv \notag
	\intertext{so that the parametric value functions
		$\sequence{\ValueFunction_t}{t \in \ic{0, \horizon-1}}$ in~\eqref{eq:parametric_value_functions}
		formulate explicitly as marginal functions:}
	\ValueFunction_{t}(\state \sep \param)
	&= \inf_{\control \in \CONTROL}
	\QFunction_t(\state, \control \sep \param) \eqsepv
	\forall (\state, \param) \in
	\STATE \times \PARAM \eqsepv
	\forall t \in \ic{0, \horizon-1}
	\eqfinp \label{eq:marginal_value_function}
	\end{align}	
\end{subequations}
For the same reason, we also introduce the (possibly empty)
parametric \emph{solution sets}
\begin{align}
\SolutionSet_t(\state, \param) &=
\argmin_{\control \in \CONTROL}
\QFunction_t(\state, \control \sep \param) \eqsepv
\forall (\state, \param) \in
\STATE \times \PARAM \eqsepv
\forall t \in \ic{0, \horizon-1}
\eqfinp \label{eq:solution_sets}
\end{align}

\subsection{The Role of the Parameter}
\label{sec:role_of_parameter}

The parameter~$\param$, as well as the control~$\control_0$,
are both (constant) decisions made at the initial time~$t=0$.
However, whereas~$\control_0$ only appears at time~$t=0$ in the initial cost function~\eqref{eq:cost_functions},
$\param$ appears in the cost functions~\eqref{eq:cost_functions}
at all times. This is why --- in a stochastic control formulation amenable
to dynamic programming --- the parameter~$p$ cannot be considered
as an initial control variable.

By contrast, it is clear from the definitions of
Problem~\eqref{eq:parametric_multistage_problem} and of the parametric value
functions $\sequence{\ValueFunction_t}{t \in \ic{0, \horizon-1}}$
in~\eqref{eq:parametric_value_functions} that the parameter $\param$
in~\eqref{eq:parameter} could be treated as a state, by introducing a new state
variable as $\statebis_t = \np{\state_t, y_t}$ for $t \in \ic{0, \horizon}$
together with a trivial --- stationary --- dynamics for its component $y$, so that
\begin{equation}
  \statebis_0 = \np{\state_0, \param} \text{ and } \quad
  \statebis_{t+1} =\np{\state_{t+1}, y_{t+1}}=
  \bp{\dynamics_t(\state_t, \control_t, \uncertain_{t+1}), y_t}
  \eqsepv \forall t \in \ic{0, \horizon-1}
  \eqfinp
  \label{eq:extended_dynamics}
\end{equation}
We have chosen not to follow that path, bearing in mind the exponential growth
of Stochastic Dynamic Programming's complexity with respect to the dimension of
the state space (termed curse of dimensionality: see \cite{bellman1957dp}).

Thus, we treat the parameter $\param$ apart from the state variables.  When the
value of $\param$ is fixed, we have seen
in~\S\ref{sec:parametric_value_functions} that we can compute
$\Phi(\param) = \ValueFunction_0(\state_0 \sep \param)$ from the knowledge of the
state functions
$\ValueFunction_1(\cdot \sep \param), \ldots, \ValueFunction_\horizon(\cdot \sep \param)$
with the Bellman induction~\eqref{eq:parametric_value_functions}.  In the
approach that we propose, we will show that, in an analogous manner, we can
compute
$\nabla\Phi(\param) = \nabla_\param\ValueFunction_0(\state_0 \sep \param)$ from the knowledge
of the state mappings
$\nabla_\param\ValueFunction_1(\cdot \sep \param), \ldots, \nabla_\param\ValueFunction_\horizon(\cdot
\sep \param)$ with a Bellman-like induction, in the spirit
of~\eqref{eq:parametric_value_functions}.
This allows us to build an efficient first-order oracle for the objective
function $\Phi$, defined as a mapping
\begin{equation}
  \param \mapsto \np{\Phi(p), \nabla\Phi(p)} \eqfinp
  \label{eq:oracle}
\end{equation}
This mapping returns the value of the objective function $\Phi$ together with the
value of the gradient of the objective (when it exists) to apply iterative
optimization steps for solving the
PMSOP~\eqref{eq:parametric_multistage_problem}.  The theoretical backbone of
such an oracle --- conditions for the differentiability of $\Phi$, formulas to
compute $\nabla \Phi$ --- is the topic of
Sect.~\ref{sec:differentiable_convex_value_functions}.

For the sake of comparison, we also consider
in~\S\ref{sec:first_order_methods_SDDP} an alternative first-order oracle based
on Stochastic Dual Dynamic Programming (SDDP, \cite{pereira1991multi}).  Here,
the state extension~\eqref{eq:extended_dynamics}, namely
$\statebis_t = \np{\state_t, y_t}= \np{\state_t, \param}$ for
$t \in \ic{0, \horizon}$, is necessary for SDDP to compute subgradients in the
extended dual space $\STATE \times \PARAM$ which, in turn, can serve to compute a
subgradient of~$\Phi$.

We expect the computational burden of both approaches --- treating $\param$ apart
from the state or not --- to be dependent on the data of the PMSOP.  An
illustrative numerical test case in day-ahead power scheduling is presented
in~Sect.~\ref{sec:day_ahead_problem}.

\section{Gradient of a Convex Parametric Value Function}
\label{sec:differentiable_convex_value_functions}

This section is organized as follows.
\begin{itemize}
\item First, in \S\ref{sec:marginal_functions},
  we review the differentiability properties
  that can be preserved by the marginal operation
  in the recursive definition of parametric value functions
  in~\eqref{eq:marginal_parametric_value_functions}.
  Indeed, we recall that differentiability
  properties of the objective function
  $\Phi$ in~\eqref{eq:parametric_multistage_problem}
  are possibly inherited from those of the parametric value functions
  $\na{\ValueFunction_t}_{t \in \ic{0, \horizon}}$
  defined by the Bellman equation in~\eqref{eq:parametric_value_functions}
  through $\Phi = \ValueFunction_{0}(\state_{0} \sep \cdot)$
  in~\eqref{eq:Phi_value_function}, under the discrete white noise
  Assumption~\ref{as:discrete_white_noise}.
\item Second, in \S\ref{sec:smooth_convex_problems},
  we show that, under suitable assumptions,
  parametric differentiability properties
  are preserved by the Bellman backward induction
  (Theorem~\ref{th:differentiability}),
  and we give formulas to compute
  the sequence of gradients
  $\nabla_\param \ValueFunction_\horizon, \ldots, \nabla_\param \ValueFunction_t,
  \ldots, \nabla_\param \ValueFunction_0$
  for a PMSOP (Theorem~\ref{th:smooth_convex_problem}).
  Reverting to Problem~\eqref{eq:parametric_multistage_problem}, the effective
  domain of the function~$\Phi$ in~\eqref{eq:parametric_multistage_problem},
  denoted by~${\cal P}$, is derived from the data of the
  PMSOP~\eqref{eq:multistage_criterion}--\eqref{eq:constraints_end} and the
  function $\Phi$ is shown to be differentiable on~$\interior\Param$.  It follows
  that, for a closed convex subset $\ParamAd \subset \interior\Param$,
  Problem~\eqref{eq:parametric_multistage_problem} is a convex differentiable
  optimization problem.
\item Third, in \S\ref{sec:lower_smooth}, we drop the differentiability
  assumptions that we made on the data of the
  PMSOP~\eqref{eq:multistage_criterion}--\eqref{eq:constraints_end}.  By a
  regularization procedure based on the Moreau envelope, we build a convex
  differentiable PMSOP over~$\PARAM$ which approximates
  Problem~\eqref{eq:multistage_criterion}--\eqref{eq:constraints_end}.  Thus,
  using the same steps as in \S\ref{sec:smooth_convex_problems}, we obtain that,
  for any closed convex subset $\ParamAd \subset \PARAM$, our regularized
  approximation of Problem~\eqref{eq:parametric_multistage_problem} is a convex
  differentiable optimization problem.
\item Lastly, in~\S\ref{sec:convergence_properties}, we show that our
  regularized approximation of Problem~\eqref{eq:parametric_multistage_problem}
  provides lower bounds converging to the value of the original problem when the
  regularization coefficient converges to zero.
\end{itemize}

\subsection{Parametric Differentiability of Marginal Functions}
\label{sec:marginal_functions}

The difficulty in studying differentiability
properties of the functions $\na{\ValueFunction_t}_{t \in \ic{0, \horizon}}$
in~\eqref{eq:parametric_value_functions}
with respect to the parameter $\param$ arises from
\begin{itemize}
\item[$(i)$] the recursive structure of
  the Bellman equations with respect to
  the time step $t$ in~\eqref{eq:parametric_value_functions},
\item[$(ii)$] the marginal operation
  with respect to the control $\control$
  in~\eqref{eq:Q_function}.
\end{itemize}

In order to address the recursion $(i)$, we adopt a standard proof scheme in
stochastic dynamic programming, by showing in \S\ref{sec:smooth_convex_problems}
that some differentiability properties are preserved by the Bellman backward
recursion~\eqref{eq:parametric_value_functions}.

Yet, before moving to the dynamic programming principle,
we need to identify what differentiability properties
can be back-propagated from
$\ValueFunction_\horizon$ to $\ValueFunction_0$.
We outline the difficulties encountered in $(ii)$
by considering a static version of the marginal operation in~\eqref{eq:Q_function}:
let the parametric marginal function
\begin{equation}
  V(\state \sep \param) = \inf_{\control \in \CONTROL}
  Q(\state, \control \sep \param) \eqfinv
  \label{eq:marginal}%
\end{equation}
be defined after a function
$Q : \STATE \times \CONTROL \times \PARAM \to \Val$.
We review some results on the directional differentiability, the subdifferentiability,
and the (Fr\'echet and G\^ateaux) differentiability of the marginal
function $V$ with respect to the parameter $\param$.

\paragraph{Directional differentiability.} A standard approach is to refer to
Danskin's Theorem to relate, under regularity assumptions, the directional
derivatives with respect to $p$ of the functions $V$ and $Q$. In the context of
multistage linear stochastic programming, this idea is exploited in
\cite{cen:inria-00579668, tercca2020envelope, guigues2023duality}.

These references mostly address the sensitivity analysis of $\Phi$
in~\eqref{eq:parametric_multistage_problem} with respect to the parameter, for
which directional derivatives seem to be appropriate.
We also report that, in~\cite[Corollary 4.4]{cen:inria-00579668},
the authors deduce from Danskin's Theorem
that the function $\Phi$ is
(Fr\'echet) differentiable
almost everywhere over the parameter space $\PARAM$.
However, in our case,
we would like to minimize $\Phi$ over
the compact set $\ParamAd$ in~\eqref{eq:parameter_constraint_set}
with a first-order method \cite{beck2017first},
which typically
requires $\Phi$ to be (sub)differentiable
everywhere over an open set $\Param \subseteq \PARAM$ containing $\ParamAd$.

\paragraph{Subdifferentiability.}
When $Q$ is a proper, convex and lower semicontinuous
function,
further results characterizing of the subdifferential of $V$ are known: for
instance, when one exists, a subgradient of $V$ at
$(\bar{\state}, \bar{\param}) \in \STATE \times \PARAM$ can be derived from a dual
solution of
\begin{equation}
  \min_{(\state, \control, \param)
    \in \dom Q }
  Q(\state, \control \sep \param) \eqsepv
  \text{s.t. } (\state, \param) = (\bar{\state}, \bar{\param}) \eqfinv
  \label{eq:primal-dual}%
\end{equation}
as documented e.g. in \cite{guigues2020inexact}, where the computation of
approximate subgradients from inexact primal-dual solutions
of~\eqref{eq:primal-dual} is also discussed.
Although used extensively in the SDDP literature \cite{pereira1991multi,
  shapiro2011analysis}, this approach requires to solve~\eqref{eq:primal-dual}
over the primal space $\STATE \times \CONTROL \times \PARAM$: when the parameter can be
treated separately, we wish to avoid this strategy as it increases the computing
costs for solving~\eqref{eq:primal-dual}.

Moreover, it would be nice that, knowing
a subgradient of the function~$Q$, we could deduce
a subgradient for the function~$V$
--- as we appeal to recursive formulas, see $(i)$.
Unfortunately, subdifferentials
do not quite meet this requirement:
given a solution $\control\opt \in \CONTROL$
of~\eqref{eq:marginal}
at $(\bar{\state}, \bar{\param}) \in \STATE \times \PARAM$,
we only have the following inclusions
\begin{equation}
  \partial_\param V(\bar{\state} \sep \bar{\param})
  \subseteq \proj{\PARAM}{\partial_{(\control, \param)} Q (\bar{\state}, \control\opt \sep \bar{\param})}
  \subseteq \partial_\param Q(\bar{\state}, \control\opt \sep \bar{\param})
\end{equation}
between subdifferentials
and, in general, these inclusions can be strict
(see \cite[Corollary 2.3.6, Examples 2.3.7 and 2.3.8]{le2021subdifferentiability}).
Even if a full characterization of the subdifferential
$\partial_\param V (\bar{\state} \sep \bar{\param})$
can be found in~\cite[Lemma 2.1]{guigues2016convergence}
under strong regularity conditions on~$Q$,
it is unfortunate that these regularity conditions
are not preserved by the marginal function~$V$ in~\eqref{eq:marginal}, so that it cannot be used to characterize the
sequence of subdifferentials
$\partial_\param \ValueFunction_\horizon, \ldots, \partial_\param \ValueFunction_t,
\ldots, \partial_\param \ValueFunction_0$.

\paragraph{Differentiability.}
The case of differentiable convex functions
is much more favorable. Formally, we concentrate
on the following classes of functions.

\begin{definition}[$\Gamma$]
  \label{de:Gamma}
  For an Euclidean space~$\YY$,
  we denote by $\Gamma\nc{\YY,\PARAM}$ the set of lower semicontinuous (lsc) convex
  functions $\gamma : \YY \times \PARAM \to \Val$.
\end{definition}

\begin{definition}[$\Theta$]
  \label{de:Theta}
  For an Euclidean space~$\YY$ and a set~$\Param \subset \PARAM$,
  we denote by $\Theta\nc{\YY,\Param}$ the subset of
  functions $\theta$ in $\Gamma\nc{\YY,\PARAM}$
  satisfying
  \begin{enumerate}
  \item
    the effective domain of $\theta$ is a (possibly empty) product set:
    $\dom \theta = \Pairxu_\theta \times \Param \subset \YY \times \PARAM$,
  \item
    for all $y \in \Pairxu_\theta$, the function
    $\theta(y, \cdot)$ is differentiable on~$\interior\Param$.
  \end{enumerate}
\end{definition}

We recall that, in the context of Definition~\ref{de:Theta}, the notions of
Fr\'{e}chet and G\^{a}teaux differentiability coincide \cite[Corollary
17.44]{Bauschke-Combettes:2017}.  To handle marginal functions as
in~\eqref{eq:marginal}, we are interested in compacity properties when
minimizing with respect to controls $\control \in \CONTROL$.

\begin{definition}[$\Gamma_\mathcal{K}$, $\Theta_\mathcal{K}$, compacity]
  \label{de:compacity}
  We denote by $\Gamma_\mathcal{K}\nc{\STATE\times\CONTROL,\PARAM}$
  the subset of functions $\gamma$ in $\Gamma\nc{\STATE\times\CONTROL,\PARAM}$
  for which there exists a compact set $\mathcal{K}_\gamma \subset \CONTROL$
  such that
  \begin{subequations}
    \begin{equation}
      \forall  (\state, \param) \in \STATE{\times}\PARAM \eqsepv
      \dom \gamma(x,\cdot, p) \subset \mathcal{K}_\gamma
      \eqfinp
    \end{equation}
    We also introduce the set
    \begin{equation}
      \Theta_\mathcal{K}\nc{\STATE\times\CONTROL,\Param}
      = \Theta\nc{\STATE\times\CONTROL,\Param} \cap \Gamma_\mathcal{K}\nc{\STATE\times\CONTROL,\PARAM}
      \eqfinv
    \end{equation}
  \end{subequations}
  for a given set $\Param \subset \PARAM$.
\end{definition}

Having introduced appropriated sets of functions, we now turn to appraise how
Stochastic Dynamic Programming maps functions from one set to another.  In
particular, regarding our discussion on the marginal operation
in~\eqref{eq:marginal}, we will see that, if $Q$ belongs to
$\Theta_\mathcal{K}\nc{\STATE\times\CONTROL,\Param}$, then $V$ belongs to
$\Theta \nc{\STATE,\Param}$ --- see Theorem~\ref{th:differentiability} and its proof
summary in Figure~\ref{fig:notations}.

\subsection{Differentiability
  and Gradient Computation by Backward Induction}
\label{sec:smooth_convex_problems}

We provide assumptions on the data of  Problem~\eqref{eq:parametric_multistage_problem}
to enforce the differentiability of the corresponding value function~$\Phi$,
from the differentiability of the parametric value functions
$\na{\ValueFunction_t}_{t \in \ic{0, \horizon}}$
in~\eqref{eq:parametric_value_functions}.
In our main result, we also
introduce a Bellman-like backward recursion
to compute the gradients
$\na{\nabla_\param \ValueFunction_t}_{t \in \ic{0, \horizon}}$,
and thus the gradient $\nabla \Phi$ of the objective function~$\Phi$ defined in Equation~\eqref{eq:upstream_problem}.

\subsubsubsection{Parametric Differentiability is Preserved
  by the Bellman Equation}

In the next theorem, we show that the set~$\Theta\nc{\STATE,\Param}$
of functions in Definition~\ref{de:Gamma} is stable by
the Bellman operator, and we give a formula to propagate gradients.

\begin{theorem}
  \label{th:differentiability}
  Let $ \va{\Uncertain}$ be a random variable taking a finite number of values.
  Given a function $\coutint : \STATE \times \CONTROL
  \times \UNCERTAIN \times \PARAM \to \Val$
  and a mapping
  $\dynamics : \STATE \times \CONTROL \times \UNCERTAIN \to \STATE$,
  we consider the Bellman operator~$\mathcal{B}$ defined, for any
  function~$\varphi : \STATE \times \PARAM \to \Val$, by
  \begin{align}
    {\cal B}\nc{\varphi} (\state,\param) =
    \inf_{\control\in\CONTROL}
    \Besp{\coutint(\state,\control,\va{\Uncertain},\param)
    + \varphi
    \bp{\dynamics(\state,\control,\va{\Uncertain}), \param}}
    \eqsepv
    \forall (\state, \param) \in \STATE\times \PARAM
    \eqfinp
    \label{eq:Bellman_Operator}%
  \end{align}
  Assume that the mapping
  $\dynamics$ is affine in~$\np{\state,\control}$.
  For any $\uncertain$ in the (finite) support of the random variable~$\va{\Uncertain}$,
  we denote $\coutint_{\uncertain}=L(\cdot,\cdot,\uncertain,\cdot)$.
  We have that,
  \begin{enumerate}
  \item if
  $\coutint_{\uncertain} \in
  \Gamma_\mathcal{K}\nc{\STATE \times \CONTROL,\PARAM}$
  for all $\uncertain \in \Support{\rvUncertain}$,
    then the set $\Gamma\nc{\STATE,\PARAM}$ is stable by the Bellman operator~$\mathcal{B}$,
  \item
    if,    for a given set $\Param \subset \PARAM$,
    $\coutint_{\uncertain} \in
    \Theta_\mathcal{K}\nc{\STATE \times \CONTROL,\Param}$
    for all $\uncertain \in \Support{\rvUncertain}$,
    then the set $\Theta\nc{\STATE,\Param}$ is stable by the Bellman operator~$\mathcal{B}$.
    Moreover, 
    for all $(\state, \param) \in \dom\bp{{\cal B}\nc{\varphi}}$
    with $\param\in\interior\Param$, the gradient of the function
    \( {\cal B}\nc{\varphi} (\state,\cdot) \) at $\param$
    is given by
    \begin{equation}
      \label{eq:Bellman_gradient}%
      \nabla_\param {\cal B}\nc{\varphi} (\state,\param) =
      \Besp{\nabla_\param\coutint(\state,\control\opt,\va{\Uncertain},\param)
        + \nabla_\param\varphi
        \bp{\dynamics(\state,\control\opt,\va{\Uncertain}), \param}}
      \eqfinv
    \end{equation}
    for any $\control\opt \in \CONTROL$
    in the nonempty argmin set of~\eqref{eq:Bellman_Operator}.
  \end{enumerate}
\end{theorem}
\begin{proof}
  We first concentrate on the proof of claim~\textit{2.},
  and we will treat claim~\textit{1.} at the end.

  $(i)$ To begin with, for any function~$\varphi : \STATE\times\PARAM\to\Val$,
  we define
  \begin{align}
    \label{eq:defBtilde}
    \mathcal{Q}\nc{\varphi} : (\state,\control, \param) \mapsto
    \Besp{\coutint(\state,\control,\va{\Uncertain},\param)
    + \varphi
    \bp{\dynamics(\state,\control,\va{\Uncertain}), \param}}
    \eqfinv
  \end{align}
  --- that is, Equation~\eqref{eq:Bellman_Operator}
  before taking the infimum ---
  and, assuming that $\varphi \in \Theta\nc{\STATE,\Param}$ and
  $\na{\coutint_{\uncertain}}_{\uncertain \in \Support{\rvUncertain}}
  \subset \Theta_\mathcal{K}\nc{\STATE \times \CONTROL,\Param}$,
  we prove that the function $\mathcal{Q}\nc{\varphi}$
  belongs to $\Theta_\mathcal{K}\nc{\STATE\times\CONTROL,\Param}$.

  First, we consider the function $Q_{\coutint} =
  \Besp{\coutint(\cdot,\cdot,\va{\Uncertain},\cdot)} :
  \STATE\times\CONTROL\times \PARAM \to \Val $.
  As,  by assumption, $\coutint_{\uncertain} \in
  \Theta\nc{\STATE\times\CONTROL,\Param}$
  for all $\uncertain$
  in the support of~$\va{\Uncertain}$, we get that
  $Q_{\coutint}$ is convex lsc, with
  \( \dom \coutint_{\uncertain} = \Pairxu_{\coutint_\uncertain} \times \Param\) for some
  $\Pairxu_{\coutint_\uncertain}\subset\STATE\times\CONTROL$.
  We deduce that the effective domain of the function~$Q_{\coutint}$
  is given by
  \begin{equation*}
  	\dom Q_{\coutint} = \bigcap_{\uncertain\in\Support{\va{\Uncertain}}}
  	\dom Q_{\coutint_{\uncertain}} =
  	\Bp{  \bigcap_{\uncertain\in\Support{\va{\Uncertain}}}
  		\Pairxu_{\coutint_\uncertain}}\times \Param =
  	\Pairxu_{Q_{\coutint}} \times \Param
  	\eqfinv
  \end{equation*}
  with $\Pairxu_{Q_{\coutint}} =
  \bigcap_{\uncertain\in\Support{\va{\Uncertain}}} \Pairxu_{\coutint_\uncertain}$.
  For any~$(\state,\control)\in\Pairxu_{Q_{\coutint}}$,
  the differentiability
  of the function~$Q_{\coutint}(\state,\control,\cdot)$
  is a straightforward consequence of the assumption that
  $\coutint_{\uncertain} \in\Theta_\mathcal{K}\nc{\STATE\times\CONTROL,\Param}$
  for all $\uncertain\in\Support{\va{\Uncertain}}$.
  Indeed, the function~$Q_{\coutint}(\state,\control,\cdot)$ is
  a finite convex combination of the
  functions~$L_{\uncertain}(\state,\control,\cdot)$,
  that are differentiable in $\interior \Param$.
  We thus obtain that $Q_{\coutint} \in \Theta\nc{\STATE\times \CONTROL,\Param}$.
  Moreover, with the notations of Definition~\ref{de:compacity}, as $\coutint_{\uncertain} \in
  \Theta_\mathcal{K}\nc{\STATE\times\CONTROL,\Param}$ for all
  $\uncertain\in\Support{\va{\Uncertain}}$, we get that
  \begin{align*}
  	\nset{ \control \in  \CONTROL}{ \exists (\state,\control, \param)\in  \Pairxu_{Q_{\coutint}} \times \Param }
  	\subset& \bigcap_{\uncertain\in\Support{\va{\Uncertain}}}
  	\nset{ \control \in  \CONTROL}{ \exists (\state,\control, \param) \in
  		\Pairxu_{\coutint_\uncertain} \times \Param} \eqfinv
  	\\
  	\subset& \bigcap_{\uncertain\in\Support{\va{\Uncertain}}}
  	\mathcal{K}_{\coutint_{\uncertain}} \eqfinv
  \end{align*}
  which implies that the function~$Q_{\coutint}$ belongs to $\Theta_\mathcal{K}\nc{\STATE\times \CONTROL,\Param}$.

  Second, we consider the function
  $Q_{\varphi}= \Besp{\varphi\bp{\dynamics(\cdot,\cdot,\va{\Uncertain}),\cdot}} :
  \STATE\times\CONTROL\times \PARAM \to \Val$.
  From the assumptions that $\varphi \in \Theta\nc{\STATE,\Param}$
  and~$\dynamics$ is affine,
  we get that the function~$Q_{\varphi}$ is convex lsc and
  that its effective domain is a Cartesian product\footnote{%
    Had the mapping~$\dynamics$ depended on the parameter~$\param$, we could
    not have concluded that the effective domain of the function~$Q_{\varphi}$
    be a Cartesian product.}.
  The differentiability of the function~$Q_{\varphi}$
  with respect to~$\param$ is straightforward.

  Gathering the results obtained for~$Q_{\coutint}$ and~$Q_{\varphi}$,
  since $\mathcal{Q}\nc{\varphi}=Q_{\coutint} + Q_{\varphi}$ and
  using the fact that $\dom (Q_{\coutint} + Q_{\varphi}) \subset \dom Q_{\coutint}$,
  we obtain that~$\mathcal{Q}\nc{\varphi}$ belongs
  to $\Theta_\mathcal{K}\nc{\STATE\times \CONTROL,\Param}$.

  \medskip

  $(ii)$ We have just proven that the function $\mathcal{Q}\nc{\varphi}$
  belongs to $\Theta_\mathcal{K}\nc{\STATE\times \CONTROL,\Param}$
  when $\varphi \in \Theta\nc{\STATE,\Param}$
 and that $\coutint_{\uncertain} \in \Theta_\mathcal{K}\nc{\STATE \times
   \CONTROL,\Param}$
   for all $\uncertain$ in the support of~$\va{\Uncertain}$.
  %
  %
  Now, there remains 
  to prove that the function~$\mathcal{B}\nc{\varphi}$ belongs to
  $\Theta\nc{\STATE,\Param}$.

  $\bullet$ First, we prove that the function~$\mathcal{B}\nc{\varphi}$ is convex lsc.
  Since, by definition,
  \begin{equation}
    \mathcal{B}\nc{\varphi}(\state, \param)
    = \inf_{\control \in \CONTROL}
    \mathcal{Q}\nc{\varphi}(\state, \control, \param) \eqsepv
    \forall (\state, \control) \in \STATE \times \CONTROL \eqfinv
  \end{equation}
  we deduce that the function~$\mathcal{B}\nc{\varphi}$ is convex
  as the marginal of a convex function.
  Then, let $K_{\mathcal{Q}\nc{\varphi}} \subset \CONTROL$
  be a compact set associated with~$\mathcal{Q}\nc{\varphi}$
  as in Definition~\ref{de:compacity}. For all
  $(\state, \control) \in \STATE \times \CONTROL$, we have that
  \begin{equation}
    \mathcal{B}\nc{\varphi}(\state, \param)
    = \inf_{\control \in \dom \np{\mathcal{Q}\nc{\varphi}}(\state,\cdot, \param)}
    \mathcal{Q}\nc{\varphi}(\state, \control, \param)
    = \inf_{\control \in K_{\mathcal{Q}\nc{\varphi}}}
    \mathcal{Q}\nc{\varphi}(\state, \control, \param)
    \label{eq:Bphi-lsc}
    \eqfinp
  \end{equation}
  Using the last expression in Equation~\eqref{eq:Bphi-lsc} and \cite[Lemma 1.30]{Bauschke-Combettes:2017},
  we get that the function~$\mathcal{B}\nc{\varphi}$ is lsc. Moreover, we obtain that
  the infimum above --- hence in~\eqref{eq:Bellman_Operator} ---
  is attained, and   we deduce that
  \begin{equation}
    \label{eq:BBtilde}%
    \forall (\state, \param) \in \STATE\times \Param \eqsepv
    \exists \control\opt \in \CONTROL \eqsepv
    \mathcal{B}\nc{\varphi}(\state, \param)
    = \mathcal{Q}\nc{\varphi}(\state, \control\opt, \param)
    \eqfinp
  \end{equation}
As a consequence, the function~$\mathcal{B}\nc{\varphi}$ never takes the value~$-\infty$
  (as so does $\mathcal{Q}\nc{\varphi}$).

  $\bullet$ Second, we prove the effective domain rectangular property (in
    Definition~\ref{de:Theta}) of
  the function~$\mathcal{B}\nc{\varphi}$.
  If $\dom\bp{\mathcal{B}\nc{\varphi}} = \emptyset$,
  the effective domain is rectangular.
  %
  We now consider the case where $\dom\bp{\mathcal{B}\nc{\varphi}} \neq \emptyset$.
  Let $(\state, \param) \in \dom\bp{\mathcal{B}\nc{\varphi}}$
  and $\control\opt \in \CONTROL$ be a minimizer
  of~\eqref{eq:Bellman_Operator} at $(\state, \param)$.
  Then, by~\eqref{eq:BBtilde}, we get that
  $\mathcal{Q}\nc{\varphi}(\state,\control\opt,\param)= \mathcal{B}\nc{\varphi}\np{\state,\param} <+\infty$
  since $(\state, \param) \in \dom\bp{\mathcal{B}\nc{\varphi}}$, hence we deduce
that $(\state,\control\opt,\param) \in \dom \np{\mathcal{Q}\nc{\varphi}}$.
  Now, using the fact that
  $\mathcal{Q}\nc{\varphi} \in \Theta\nc{\STATE\times\CONTROL,\Param}$, we also have that
  $\dom \np{\mathcal{Q}\nc{\varphi}}$ is a product and moreover, as just seen, the product is nonempty.
  Consider now $\param'\in \Param$. By the product property of the effective domain of~$\mathcal{Q}\nc{\varphi}$, we have that
  $(\state,\control\opt,\param') \in \dom \mathcal{Q}\nc{\varphi}$. Using the fact that
  $\mathcal{B}\nc{\varphi}\np{\state,\param'} \le \mathcal{Q}\nc{\varphi}(\state,\control\opt,\param') < +\infty$,
  we obtain that $\param' \in \dom\bp{\mathcal{B}\nc{\varphi}(\state, \cdot)}$, finally giving that
  $\dom\bp{\mathcal{B}\nc{\varphi}(\state, \cdot)} = \Param$.
  This proves that
  $\dom\bp{\mathcal{B}\nc{\varphi}} = \State \times \Param$,
  for some set $\State \subset \STATE$.
  \medskip

  $\bullet$ Third, we turn to the parametric differentiability of
  the function~$\mathcal{B}\nc{\varphi}: \STATE \times \PARAM \to \Val$.
  The only relevant case is when $\dom\bp{\mathcal{B}\nc{\varphi}} \neq \emptyset$.
  For this purpose, we consider a fixed $\state \in \State$ (where  $\dom\bp{\mathcal{B}\nc{\varphi}} = \State \times \Param$),
  and we introduce the two functions
  \begin{equation}
    \psi_\state = \mathcal{B}\nc{\varphi}(\state, \cdot)
    \quad \text{ and } \quad
    \Psi_\state =
    \mathcal{Q}\nc{\varphi}(\state, \cdot, \cdot)
    \eqfinp
  \end{equation}
  Now, the parametric differentiability of
  the function~$\mathcal{B}\nc{\varphi}$ boils down to that
of the function~$\psi_\state $.
As a preliminary result,    we show that  \( \partial_\param  \Psi_\state(\control, \param)
= \ba{ \nabla_\param \Psi_\state \np{\control, \param} } \) for any
\( \np{\control, \param} \in \CONTROL \times \interior{\Param} \) such that
\( \np{\state, \control, \param} \in \dom \np{\mathcal{Q}\nc{\varphi}}\).
  Indeed, as the function $\mathcal{Q}\nc{\varphi}$
  belongs to the set $\Theta_\mathcal{K}\nc{\STATE\times \CONTROL,\Param}$, we obtain two results:
  on the one hand, the function $\Psi_\state : \CONTROL \times \PARAM \to \Val$
  is differentiable with respect to its second argument~$\param \in \interior{\Param}$;
  on the other hand, as the function $\Psi_\state : \CONTROL \times \PARAM \to \Val$
is proper, convex and lsc, we deduce that it is
subdifferentiable with respect to its second argument~$\param \in \interior{\Param}$,
by \cite[Proposition 16.27]{Bauschke-Combettes:2017}.
Thus, by \cite[Proposition 17.31 (i)]{Bauschke-Combettes:2017},
we conclude that  \( \partial_\param  \Psi_\state(\control, \param)
= \ba{ \nabla_\param \Psi_\state \np{\control, \param} } \), where both objects are well-defined
for any
\( \np{\control, \param} \in \CONTROL \times \interior{\Param} \) such that
\( \np{\state, \control, \param} \in \dom \np{\mathcal{Q}\nc{\varphi}}\).
Now, as said above, we study the differentiabilty of the function~$\psi_\state $.
For this purpose, we consider $\param \in \interior{\Param}$
and $\control\opt \in \CONTROL$ a minimizer as in~\eqref{eq:BBtilde}.
  By \cite[Corollary 2.63]{Mordukhovich-2013}, we get that
  the function~$\psi_\state $ is is proper, convex, lsc, subdifferentiable and that
  \( \partial\psi_\state(\param) =
  \defset{s\in\PARAM }{\np{0, s} \in \partial\Psi_\state(\control\opt, \param)} \).
  It is easily proved that the projection of~\( \partial\Psi_\state(\control\opt, \param) \)
on its second component is included in~\( \partial_\param\Psi_\state(\control\opt, \param) \).
  As we have just proven that
  \(  \partial_\param\Psi_\state(\control\opt, \param)
  = \ba{ \nabla_\param \Psi_\state \np{\control\opt, \param} } \),
  we deduce that the nonempty set
  $\partial \psi_\state(\param)$ is a singleton.
  By \cite[Proposition 17.31 (ii)]{Bauschke-Combettes:2017},
  we conclude that the function~$\psi_\state$ is differentiable at~$\param$
  and that
  $\nabla \psi_\state(\param)
  = \nabla_\param \Psi_\state \np{\control\opt, \param}$, that is,
  \begin{equation}
    \label{eq:Bellman_gradient_before}%
    \nabla_\param {\cal B}\nc{\varphi} (\state,\param) =\nabla_\param
    \Besp{\coutint(\state,\control\opt,\va{\Uncertain},\param)
      + \varphi\bp{\dynamics(\state,\control\opt,\va{\Uncertain}),\param}}
    \eqfinp
  \end{equation}

  This finally proves that
  $\mathcal{B}\nc{\varphi} \in \Theta\nc{\STATE,\Param}$, that is,
  the set~$\Theta\nc{\STATE,\Param}$ is stable by
  the Bellman operator~$\mathcal{B}$. Moreover, since the support
  of~$\va{\Uncertain}$  is a finite set, exchanging the derivation
  and expectation operators in~\eqref{eq:Bellman_gradient_before}
  is trivial, so that~\eqref{eq:Bellman_gradient} holds true.
  \medskip

  As for claim~\textit{1.},
  when
\( \coutint_{\uncertain} \in \Gamma_\mathcal{K}\nc{\STATE \times
  \CONTROL,\PARAM} \)
     for all $\uncertain$ in the support of~$\va{\Uncertain}$,
  and $\varphi \in \Gamma\nc{\STATE,\PARAM}$,
  we obtain with analogous arguments that
  $(i)$
  the function $\mathcal{Q}\nc{\varphi}$
  belongs to the set $\Gamma_\mathcal{K}\nc{\STATE\times\CONTROL,\PARAM}$,
  $(ii)$ the function $\cal B\nc{\varphi}$
  belongs to the set $\Gamma\nc{\STATE,\PARAM}$.
\end{proof}

Figure~\ref{fig:notations} illustrates the links induced by the Bellman
operator~$\mathcal{B}$ between the different subsets of functions, as revealed
by Theorem~\ref{th:differentiability}. The operator~$\mathcal{Q}$
is defined in \eqref{eq:defBtilde} and corresponds to the operator
$\mathcal{B}$ before taking the infimum in~$\control$.

\begin{figure}[htpb]
  \centering
  {\includegraphics[width=0.8\linewidth]{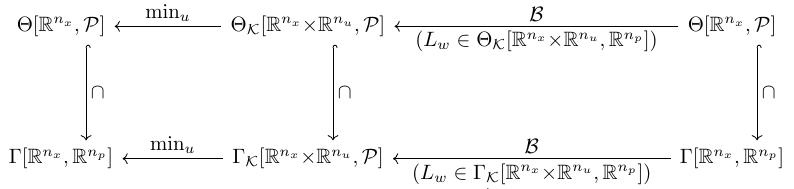}}
  \caption{Stability of the classes of functions~$\Theta$
    (Definition~\ref{de:Theta})
    and~$\Gamma$ (Definition~\ref{de:Gamma})
    by the Bellman operator~$\mathcal{B}$ (Equation~\eqref{eq:Bellman_Operator})
  \label{fig:notations}}
\end{figure}

\subsubsubsection{Application to Parametric Value Functions}
We introduce assumptions on the data
of Problem~\eqref{eq:parametric_multistage_problem}
to enforce the convexity and the parametric differentiability
of the parametric value functions
$\na{\ValueFunction_t}_{t \in \ic{0, \horizon}}$
in~\eqref{eq:parametric_value_functions}.

\begin{assumption}[convex multistage problem]
  \label{as:convexity}%
  We assume that
  \begin{enumerate}
  \item Problem~\eqref{eq:parametric_multistage_problem}
    is feasible,
  \item the dynamics $\na{\dynamics_t}_{t \in \ic{0,T{-}1}}$
    in~\eqref{eq:dynamics_ch2}
    are affine with respect to their arguments $(\state, \control)$,
  \item
    for all $t \in \ic{0,T{-}1}$,
    the functions 
$\coutint_t(\cdot, \cdot, \uncertain, \cdot)$
    ---defined after the stage cost $\coutint_t$ in~\eqref{eq:parametric_stage_cost}---
    belong to $\Gamma_\mathcal{K}\nc{\STATE\times \CONTROL,\PARAM}$,
        for all $\uncertain$ in the support of~$\va{\Uncertain}$,
  \item the final cost $\Final$
    in~\eqref{eq:parametric_final_cost} belongs to $\Gamma\nc{\STATE,\PARAM}$.
  \end{enumerate}
\end{assumption}

Assumption~\ref{as:convexity} puts us in a standard context
for convex stochastic multistage optimization problems
(see e.g. \cite[assumption $H_1$]{girardeau2015convergence}).
By assuming that Problem~\eqref{eq:parametric_multistage_problem}
is feasible, we obtain that the function~$\ValueFunction_0$ is proper
--- as $\Phi = \ValueFunction_0(\state_0 \sep \cdot)$ ---
and therefore that all parametric value functions
$\na{\ValueFunction_t}_{t \in \ic{0, \horizon}}$ are proper
--- due to the Bellman equations~\eqref{eq:parametric_value_functions}.
At a lower level of detail, this can also be enforced
by a ``relatively complete recourse'' assumption
(see e.g \cite{girardeau2015convergence, leclere2020exact}).

We make a second assumption to handle parametric differentiability.	

\begin{assumption}[parametric differentiability]
  \label{as:smoothness}%
  Let~$\Param$ be a given subset of~$\PARAM$. We assume that
  \begin{enumerate}
  \item for all $t \in \ic{0,T{-}1}$,
    the functions 
$\coutint_t(\cdot, \cdot, \uncertain, \cdot)$
    ---defined after the stage cost $\coutint_t$ in~\eqref{eq:parametric_stage_cost}---
    belong to $\Theta_\mathcal{K}\nc{\STATE\times \CONTROL,\Param}$,
    for all $\uncertain$ in the support of~$\va{\Uncertain}$,
  \item the final cost $\Final$
    in~\eqref{eq:parametric_final_cost} belongs to $\Theta\nc{\STATE,\Param}$.
  \end{enumerate}
\end{assumption}

We now state our main result regarding the differentiability
of the parametric value functions
$\na{\ValueFunction_t}_{t \in \ic{0, \horizon}}$.

\begin{theorem}\label{th:smooth_convex_problem}%
  Under the discrete white noise Assumption~\ref{as:discrete_white_noise} and
  the convex multistage problem Assumption~\ref{as:convexity},
  we have that the parametric value functions
    $\sequence{\ValueFunction_t}{t \in \ic{0, \horizon}}$
    defined in~\eqref{eq:parametric_value_functions}
    are proper and belong to $\Gamma\nc{\STATE,\PARAM}$.

Moreover, under
the parametric differentiability Assumption~\ref{as:smoothness},
we get that
\begin{itemize}
  \item
    the parametric value functions
    $\sequence{\ValueFunction_t}{t \in \ic{0, \horizon}}$
    belong to $\Theta\nc{\STATE,\Param}$
    and, for all $p\in \interior{\Param}$, their gradients can be
    computed by backward induction, with, at final stage~$\horizon$,
    \begin{subequations}
      \label{eq:backward_induction}%
      \begin{align}
        \nabla_\param \ValueFunction_\horizon \np{\state \sep \param}
        &=
          \nabla_\param \Final(\state , \param)
          \eqsepv
          \forall \state \in \dom \ValueFunction_\horizon \np{\cdot \sep \param} \eqfinv
          \label{eq:value_function_gradient_T}
          \intertext{and, at any stage $t \in \ic{0, \horizon-1}$ and for all $\state \in \dom \ValueFunction_t \np{\cdot \sep \param}$,}
          \nabla_\param \ValueFunction_t \np{\state \sep \param}
        &=
          \EE
          \Bc{\nabla_\param
          \coutint_t(\state, \control\opt, \rvUncertain_{t+1}, \param)
          + \nabla_\param \ValueFunction_{t+1}
          \bp{\dynamics_t\np{\state, \control\opt, \rvUncertain_{t+1}} \sep \param}}
          \eqfinv \label{eq:value_function_gradient}%
      \end{align}
      where the control~$\control\opt$ is any control in the solution set
      $\SolutionSet_t(\state , \param)$ defined in~\eqref{eq:solution_sets},
    \end{subequations}
  \item
    for any closed convex subset~${\ParamAd} \subset \Param$,
    the upstream optimization problem~\eqref{eq:upstream_problem}
    is a convex differentiable optimization problem.
\end{itemize}
\end{theorem}

\begin{proof} For each time $t\in \ic{0,\horizon}$, we consider the Bellman operator~$\mathcal{B}_t$ defined, for any
  function~$\varphi : \STATE \times \PARAM \to \Val$ and any for all $(\state, \param) \in \STATE\times \PARAM$, by
  $$
  {\cal B}_t\nc{\varphi} (\state,\param) =
  \inf_{\control\in\CONTROL}
  \Besp{\coutint_t(\state,\control,\va{\Uncertain},\param)
    + \varphi
    \bp{\dynamics(\state,\control,\va{\Uncertain}), \param}}
  \eqfinp
  $$
  Thus,
  the value functions $\sequence{\ValueFunction_t}{t \in \ic{0, \horizon}}$
  defined in~\eqref{eq:parametric_value_functions} satisfy $\ValueFunction_{\horizon}=\Final$ and
  for all $t\in \ic{0,T{-}1}$, $\ValueFunction_{t} = {\cal B}_t \nc{\ValueFunction_{t+1}}$.
  The proof follows by applying backward in time Theorem~\ref{th:differentiability} to the Bellman operator
  ${\cal B}_t$ and the value function $\ValueFunction_{t+1}$.
\end{proof}

\begin{remark}
In the case where Problem~\eqref{eq:multistage_criterion}--\eqref{eq:constraints_end}
has explicit constraints of the form
$\rvControl_t \in \ConstraintSet_t\np{\rvState_t, \param}$,
at each time~$t$, the parametric stage cost~$\coutint_t$
in~\eqref{eq:parametric_stage_cost} incorporates the indicator
function $\indicator_{\text{gr}\np{\ConstraintSet_t}}$,
where $\text{gr}\np{\ConstraintSet_t}$ is the graph
of the set-valued mapping~$\ConstraintSet_{t}$,
defined as the set 
$\defset{(\state, \control, \param) \in \STATE\times\CONTROL\times\PARAM}{\control \in \ConstraintSet_t(\state, \param)}$
(see Remark~\ref{rm:explicit_constraints}). 
Then, a fairly natural
condition to ensure the differentiability of~$\coutint_{t}$ with
respect to the parameter~$\param$ is that~$\ConstraintSet_{t}$ does not depend on~$\param$.
\end{remark}

\begin{remark}
  Referring to our discussion
  on  the role of the parameter
  in \S\ref{sec:role_of_parameter},
  observe that in the
  backward induction~\eqref{eq:backward_induction},
  the parameter $\param$ is always fixed.
  This is in contrast with the state variable $\state$,
  whose evolution is ruled by the dynamics
  $\na{\dynamics_t}_{t \in \ic{0,T{-}1}}$
  in~\eqref{eq:dynamics_ch2}.
\end{remark}

\subsection{Regularization of Convex Nondifferentiable Parametric Value Functions}
\label{sec:lower_smooth}

We now turn to the nondifferentiable case, that is, we do not assume
anymore that the functions~$\na{\coutint_{t}}_{t\in\ic{0,\horizon{-}1}}$
and~$\coutfin$ in~\S\ref{sec:problem_formulation} are differentiable with respect to the parameter~$\param$.
To overcome this drawback and go back to the differentiable situation,
we appeal to the Morean-Yosida regularization. We recall the definition
of the Moreau envelope \cite{moreau1965proximite, yosida1971functional}.

\begin{definition}
  Let $n \in \NN^*$, $f : \RR^n \to \barRR$ be a function
  and $\mu \in \spRR$ be a regularization coefficient.
  The Moreau envelope of $f$ is the function
  \begin{equation}
    f^\mu : \RR^n \to \barRR \eqsepv z \mapsto \inf_{z' \in \RR^n}
    \Bp{f(z') + \frac{1}{2\mu}\Norm{z-z'}_2^2}
    \eqfinp
  \end{equation}
\end{definition}

We refer to \cite[Chapter~1, \S G]{rockafellar2009variational} and
\cite[Chapter~12, \S4]{Bauschke-Combettes:2017} for a review of the properties
of the Moreau envelope.  Given values of
$\np{\state, \control, \uncertain} \in \STATE \times \CONTROL \times \UNCERTAIN$ and a
regularization coefficient $\mu \in \spRR$, we introduce the \emph{parametric Moreau
  envelopes}
$\sequence{\StageCost_{t}^\mu(\state, \control, \uncertain, \cdot)} {t \in \ic{0,
    \horizon-1}}$ and $\Final^\mu(\state, \cdot)$ of the parametric cost functions
$\na{\coutint_t}_{t \in \ic{0,T{-}1}}$ in~\eqref{eq:parametric_stage_cost} and
$\Final$ in~\eqref{eq:parametric_final_cost}, with respect to the
parameter~$\param$ in~\eqref{eq:parameter}, defined as
\begin{subequations}
  \label{eq:my_cost}%
  \begin{align}
    \StageCost_{t}^\mu(\state, \control, \uncertain, \param)
    &=
      \inf_{\param'\in\PARAM}
      \Bp{\StageCost_{t}(\state, \control, \uncertain, \param')
      + \frac{1}{2\mu}\Norm{\param - \param'}^2_2}
      \eqsepv
      \forall t \in \ic{0, \horizon-1}
      \eqsepv
      \forall \param
      \in \PARAM
      \eqfinv \label{eq:my_stage_cost}
    \\
    \Final^\mu(\state, \param)
    &=
      \inf_{\param'\in\PARAM}
      \Bp{\Final(\state, \param') +
      \frac{1}{2\mu}\Norm{\param - \param'}^2_2}
      \eqsepv
      \forall \param \in \PARAM
      \eqfinp \label{eq:my_final_cost}%
  \end{align}
\end{subequations}

In order to ensure that the regularized parametric cost
functions $\na{\StageCost_{t}^\mu}_{t \in \ic{0,\horizon{-}1}}$
and $\Final^\mu$ are lsc, we introduce the
following parametric compacity assumption.

\begin{assumption}[parametric compacity]
  \label{as:compacity}%
  We suppose that the discrete white noise
  Assumption~\ref{as:discrete_white_noise} holds true, and
  that there exists
  a compact set~$\Param \subset \PARAM$
  such that
  for all $t \in \ic{0,T{-}1}$
  and for all $\uncertain \in \Support{\rvUncertain_t}$,
  the stage costs
  $\coutint_t$ in~\eqref{eq:parametric_stage_cost}
  satisfy
  \begin{subequations}
    \begin{align}
      &\forall (\state, \control) \in \STATE{\times}\CONTROL
        \eqsepv
        \dom {\coutint_t(\state, \control, \uncertain, \cdot)} \subset \Param
        \eqfinv
      \intertext{and the final cost $\Final$
      in~\eqref{eq:parametric_final_cost} satisfies}
      &
        \forall \state \in \STATE
        \eqsepv
        \dom \Final(\state, \cdot) \subset \Param
        \eqfinp
    \end{align}
  \end{subequations}
\end{assumption}

\begin{lemma}
  \label{le:smooth_cost}
  Let $\mu > 0$.
  Under the parametric
  compacity Assumption~\ref{as:compacity},
  we have that
  \begin{enumerate}
  \item
    for all $t \in \ic{0,T{-}1}$
    and all $\uncertain \in \Support{\rvUncertain_t}$,
    if $\coutint_t(\cdot, \cdot, \uncertain, \cdot) \in \Gamma_\mathcal{K}[\STATE\times\CONTROL, \PARAM]$,
    then $\coutint_t^\mu(\cdot, \cdot, \uncertain, \cdot) \in \Theta_\mathcal{K}[\STATE\times\CONTROL, \PARAM]$,
  \item if $\Final \in \Gamma[\STATE, \PARAM]$,
    then $\Final^\mu \in \Theta[\STATE, \PARAM]$.
  \end{enumerate}
\end{lemma}

\begin{proof}
  Let $\uncertain\in \Support{\rvUncertain_t}$ be given and consider
  the function
  $\coutint_t(\cdot, \cdot, \uncertain, \cdot)$
  and its parametric Moreau envelope
  $\coutint_t^\mu(\cdot, \cdot, \uncertain, \cdot)$
  defined by~\eqref{eq:my_cost}.

  First, we prove that
  $\coutint_t^\mu(\cdot, \cdot, \uncertain, \cdot) \in \Gamma_\mathcal{K}[\STATE\times\CONTROL, \PARAM]$.
  As the marginal of a convex function, the function
  $\coutint_t^\mu(\cdot, \cdot, \uncertain, \cdot)$
  is convex.
  The infimum in~\eqref{eq:my_cost} can be
  taken equivalently over the fixed compact set
  $\Param \subset \PARAM$ given in Assumption~\ref{as:compacity}.
  As $\coutint_t(\cdot, \cdot, \uncertain, \cdot)
  \in \Gamma_\mathcal{K}[\STATE\times\CONTROL, \PARAM]$
  is lsc, we deduce that the marginal function
  $\coutint_t^\mu(\cdot, \cdot, \uncertain, \cdot)$
  is also lsc, and that the infimum in~\eqref{eq:my_cost}
  is attained \cite[Lemma~1.30]{Bauschke-Combettes:2017}.
  It follows that $\coutint_t^\mu(\cdot, \cdot, \uncertain, \cdot)$
  takes values in $\Val$ --- as so does
  $\coutint_t(\cdot, \cdot, \uncertain, \cdot)$ ---
  and therefore that
  $\coutint_t^\mu(\cdot, \cdot, \uncertain, \cdot) \in \Gamma[\STATE\times\CONTROL, \PARAM]$.
  Now, let $\control \in \CONTROL$ be such that
  there exists $(\state, \param) \in \STATE \times \PARAM$
  so that
  $(\state, \control, \param)
  \in \domain \coutint_t^\mu(\cdot, \cdot, \uncertain, \cdot)$.
  As we have just seen, there exists a minimizer
  $\param\opt \in \PARAM$ such that
  \begin{equation}
    \coutint_t^\mu(\state, \control, \uncertain, \param)
    = \coutint_t(\state, \control, \uncertain, \param\opt)
    + \frac{1}{2\mu} \Norm{\param - \param\opt}_2^2
    \eqfinp
    \label{eq:MY_equality}%
  \end{equation}
  Necessary, $(\state, \control, \param\opt) \in
  \domain \coutint_t(\cdot, \cdot, \uncertain, \cdot)$,
  which proves that $\control \in \mathcal{K}_{\coutint_t(\cdot, \cdot, \uncertain, \cdot)}$,
  following Definition~\ref{de:compacity}.
  We deduce that
  $\coutint_t^\mu(\cdot, \cdot, \uncertain, \cdot) \in \Gamma_\mathcal{K}[\STATE\times\CONTROL, \PARAM]$.

  Second, we prove that
  $\coutint_t^\mu(\cdot, \cdot, \uncertain, \cdot) \in \Theta_\mathcal{K}[\STATE\times\CONTROL, \PARAM]$.
  Following Definition~\ref{de:compacity}, we only need
  to prove that $\coutint_t^\mu(\cdot, \cdot, \uncertain, \cdot) \in \Theta[\STATE\times\CONTROL, \PARAM]$.
  Let us consider again $(\state, \control, \param)
  \in \domain\coutint_t^\mu(\cdot, \cdot, \uncertain, \cdot)$.
  The function
  $\coutint_t(\state, \control, \uncertain, \cdot)$
  is convex, lsc, never takes the value $-\infty$ (by assumption)
  and its domain is nonempty (by~\eqref{eq:MY_equality}).
  It follows that its Moreau envelope
  $\coutint_t^\mu(\state, \control, \uncertain, \cdot)$
  is finite valued everywhere on $\PARAM$ (\cite[Proposition~12.15]{Bauschke-Combettes:2017})
  and differentiable on $\PARAM$
  (\cite[Proposition~12.30]{Bauschke-Combettes:2017}).
  We deduce that $\coutint_t^\mu(\cdot, \cdot, \uncertain, \cdot) \in \Theta[\STATE\times\CONTROL, \PARAM]$,
  and finally that
  $\coutint_t^\mu(\cdot, \cdot, \uncertain, \cdot) \in \Theta_\mathcal{K}[\STATE\times\CONTROL, \PARAM]$.

  The proof that $\Final^\mu \in \Theta_\mathcal{K}[\STATE, \PARAM]$ is analogous.
\end{proof}

We are now ready to introduce the
\emph{lower smooth parametric value functions}.
For any regularization coefficient~$\mu \in~\spRR$, we define
\begin{subequations}
  \begin{align}
    \SmoothValueFunction_\horizon(\state \sep \param)
    &= \Final^\mu(\state, \param) \eqsepv
      \forall (\state, \param) \in \STATE \times \PARAM \eqfinv
      \label{eq:smooth_final}
    \\
    \SmoothValueFunction_t(\state \sep \param)
    &= \inf_{\control \in \CONTROL
      }\EE
      \Bc{\StageCost_t^\mu(\state, \control, \rvUncertain_{t+1}, \param) + \SmoothValueFunction_{t+1}
      \bp{\dynamics_t\np{\state, \control, \rvUncertain_{t+1}} \sep \param}}
      \eqsepv \notag\\
    &\hspace{4cm}
      \forall (\state, \param) \in \STATE \times \PARAM \eqsepv \forall t \in \ic{0, \horizon-1}
      \eqfinp \label{eq:marginal_smooth_value_functions}%
  \end{align}
  \label{eq:smooth_value_functions}%
\end{subequations}
The lower smooth parametric value functions
$\sequence{\SmoothValueFunction_{t}}{t \in \ic{0,\horizon}}$
have several interesting properties,
which we gather in Theorem~\ref{th:lower_smooth_properties}.

\begin{theorem}
  \label{th:lower_smooth_properties}
  Let $\mu \in \spRR$ be a regularization coefficient,
  and let $\sequence{\SmoothValueFunction_{t}}{t \in \ic{0,\horizon}}$
  be the lower smooth parametric value functions
  defined in~\eqref{eq:smooth_value_functions}.
  Under the discrete white noise Assumption~\ref{as:discrete_white_noise},
  the convex multistage problem Assumption~\ref{as:convexity},
  and the parametric compacity Assumption~\ref{as:compacity},
  \begin{enumerate}
  \item the functions
    $\sequence{\SmoothValueFunction_{t}}{t \in \ic{0,\horizon}}$
    provide lower bounds on the value functions
    $\sequence{\ValueFunction_{t}}{t \in \ic{0,\horizon}}$
    defined in~\eqref{eq:parametric_value_functions}, that is,
    \begin{equation}\label{eq:smooth_lower_bound}
      \SmoothValueFunction_t
      \leq \ValueFunction_t
      \eqsepv
      \forall t \in \ic{0, \horizon}
      \eqfinv
    \end{equation}
  \item the functions
    $\sequence{\SmoothValueFunction_{t}}{t \in \ic{0,\horizon}}$
    are proper and belong to the set $\Theta[\STATE, \PARAM]$ from Definition~\ref{de:Theta}.
    Moreover, their gradients can be computed by backward induction,
    with, at final stage~$\horizon$,
    \begin{subequations}
      \label{eq:backward_induction_Moreau}%
      \begin{align}
        \nabla_\param \SmoothValueFunction_\horizon \np{\state \sep \param}
        &=
          \nabla_\param \Final^\mu(\state,\param) \eqfinv
          \intertext{and, at any stage $t \in \ic{0, \horizon-1}$,}
          \nabla_\param \SmoothValueFunction_t \np{\state \sep \param}
        &=
          \EE
          \Bc{\nabla_\param
          \coutint_t^\mu(\state, \control\opt, \rvUncertain_{t+1},\param)
          + \nabla_\param \SmoothValueFunction_{t+1}
          \bp{\dynamics_t\np{\state,\control\opt,\rvUncertain_{t+1}} \sep \param}}
          \eqfinv \label{eq:value_function_gradient_Moreau}%
      \end{align}
      for any $\control\opt$ in the nonempty
      solution set $\SolutionSet_t(\state,\param)$
      of Problem~\eqref{eq:marginal_smooth_value_functions}
    \end{subequations}
  \end{enumerate}
\end{theorem}

\begin{proof}
  First, we prove Inequality~\eqref{eq:smooth_lower_bound}
  proceeding by backward induction.
  By the properties of the Moreau envelope,
  we have that
  $\ValueFunction_\horizon^\mu = \Final^\mu
  \leq \Final = \ValueFunction_\horizon$
  \cite[Proposition~12.9]{Bauschke-Combettes:2017}.
  Let $t \in \ic{0, \horizon-1}$
  and let us assume that~\eqref{eq:smooth_lower_bound} holds true at stage $t+1$.
  Consequently, by application of \cite[Proposition~12.9]{Bauschke-Combettes:2017},
  we have that for any
  $(\state, \control, \param)
  \in \STATE\times\CONTROL\times\PARAM$
  and $\uncertain \in \Support{\rvUncertain_{t+1}}$,
  \begin{equation*}
    \coutint_t^\mu(\state, \control, \uncertain, \param)
    + \SmoothValueFunction_{t+1}
    \bp{\dynamics_t\np{\state,\control,\uncertain} \sep \param}
    \leq
    \coutint_t(\state, \control, \uncertain, \param)
    + \ValueFunction_{t+1}
    \bp{\dynamics_t\np{\state,\control,\uncertain} \sep \param}
    \eqfinv
  \end{equation*}
  from which we deduce that, for any
  $(\state, \param) \in \STATE\times\PARAM$,
  \begin{align*}
    \SmoothValueFunction_t(\state \sep \param)
    &=
      \inf_{\control \in \CONTROL}
      \EE\Bc{
      \coutint_t^\mu(\state, \control, \rvUncertain_{t+1}, \param)
      + \SmoothValueFunction_{t+1}
      \bp{\dynamics_t\np{\state,\control,\rvUncertain_{t+1}} \sep \param}}
      \eqfinv
    \\
    &\leq
      \inf_{\control \in \CONTROL}
      \EE\Bc{
      \coutint_t(\state, \control, \rvUncertain_{t+1}, \param)
      + \ValueFunction_{t+1}
      \bp{\dynamics_t\np{\state,\control,\rvUncertain_{t+1}} \sep \param}}
      \eqfinv
    \\
    &= \ValueFunction_t(\state \sep \param) \eqfinv
  \end{align*}
  so that Inequality~\eqref{eq:smooth_lower_bound} holds true.

  Second, we apply Theorem~\ref{th:smooth_convex_problem}.
  To do this, we observe that the parametric value functions
  $\sequence{\SmoothValueFunction_{t}}{t \in \ic{0,\horizon}}$
  are the value functions of the new PMSOP
  $\min_{\param \in \ParamAd} \Phi^\mu(\param)$,
  whose definition follows the one of the original
  PMSOP~\eqref{eq:parametric_multistage_problem},
  except that the data of the problem
  $\bp{\sequence{\dynamics_{t}}{t \in \ic{0,\horizon-1}},
    \sequence{\coutint_{t}}{t \in \ic{0,\horizon-1}},
    \Final}$
  is replaced by
  $\bp{\sequence{\dynamics_{t}}{t \in \ic{0,\horizon-1}},
    \sequence{\coutint_{t}^\mu}{t \in \ic{0,\horizon-1}},
    \Final^\mu}$.
  Lemma~\ref{le:smooth_cost} tells us that
  the new components of this data triplet
  satisfies the conditions
  of Theorem~\ref{th:smooth_convex_problem},
  so that we only need to prove that
  the new PMSOP
  $\min_{\param \in \ParamAd} \Phi^\mu(\param)$
  is feasible (to fulfill Assumption~\ref{as:convexity}).
  Again, under the white noise Assumption~\ref{as:discrete_white_noise},
  we have from the dynamic programming principle
  that $\Phi^\mu = \SmoothValueFunction_{0}(\state_0 \sep \cdot)$.
  By definition of the functions
  $\na{\SmoothValueFunction_{t}}_{t\in\ic{0,\horizon}}$
  in~\eqref{eq:smooth_value_functions},
  and by the properties of the cost functions
  $\sequence{\coutint_{t}^\mu}{t \in \ic{0,\horizon-1}}$
  and $\Final^\mu$ in Lemma~\ref{le:smooth_cost},
  we obtain that all functions
  in $\na{\SmoothValueFunction_{t}}_{t\in\ic{0,\horizon}}$
  belong to $\Gamma[\STATE, \PARAM]$
  (applying Theorem~\ref{th:differentiability}
  backward in time).
  It follows that $\SmoothValueFunction_{0}$
  takes values in $\Val$.
  Now, as the original PMSOP
  \eqref{eq:parametric_multistage_problem}
  is feasible, taking $\param \in \domain\Phi \cap \ParamAd$,
  we have that
  $\Phi^\mu(\param) = \SmoothValueFunction_{0}(\state_0 \sep \param) \leq \ValueFunction_{0}(\state_0 \sep \param)
  = \Phi(\param) < +\infty$,
  which proves that the problem
  $\min_{\param \in \ParamAd} \Phi^\mu(\param)$
  is feasible too.
  Therefore, we can apply Theorem~\ref{th:smooth_convex_problem}.
  This concludes the proof.
\end{proof}

To summarize, in the nondifferentiable case, we have obtained
differentiable value functions by Theorem~\ref{th:lower_smooth_properties}
that are lower bounds of the original value functions.
The gradients of these differentiable value functions can be computed
by backward induction as stated by Theorem~\ref{th:smooth_convex_problem}.

\subsection{Convergence Properties of Regularized
  Convex Parametric Value Functions}
\label{sec:convergence_properties}

Finally, we prove some convergence properties of the lower smooth parametric
value functions $\sequence{\SmoothValueFunction_{t}}{t \in \ic{0,\horizon}}$
defined in~\eqref{eq:smooth_value_functions}, which show that they are suitable
candidates to approximate the original value functions
$\sequence{\ValueFunction_{t}}{t \in \ic{0,\horizon}}$
in~\eqref{eq:parametric_value_functions} for solving
Problem~\eqref{eq:parametric_multistage_problem}.  We refer the reader to the
definition of pointwise convergence in~\cite[\S7.A]{rockafellar2009variational},
denoted by ``$\xrightarrow[]{\textsc{p}}$'', and to the definition of
epiconvergence in~\cite[\S7.B]{rockafellar2009variational} denoted by
``$\xrightarrow[]{\textnormal{e}}$''.

\begin{proposition}
  \label{pr:convergence_of_smooth_vf}
  We suppose that the discrete white noise
  Assumption~\ref{as:discrete_white_noise} holds true.
  Let $\na{\mu_n}_{n\in\NN} \in \np{\spRR}^\NN$
  be a nonincreasing sequence of positive real numbers such that
  $\lim_{n\to+\infty} \mu_n = 0$,
  let $\na{\RawSmoothValueFunction^{\mu_n}_{t}}_{t \in \ic{0,\horizon}, n\in\NN}$
  be lower smooth parametric value functions
  as defined in~\eqref{eq:smooth_value_functions},
  and let $\na{\ValueFunction_{t}}_{t \in \ic{0,\horizon}}$
  be the parametric value functions
  defined in~\eqref{eq:parametric_value_functions}.

  Under the discrete white noise Assumption~\ref{as:discrete_white_noise},
  the convex multistage problem Assumption~\ref{as:convexity}
  and the parameteric compacity Assumption~\ref{as:compacity},
  we have the following convergence property for all $t \in \ic{0,\horizon}$:
  \begin{equation}
      \RawSmoothValueFunction^{\mu_n}_{t}
      \xrightarrow[n\to+\infty]{\textnormal{e}} \ValueFunction_{t}
    \eqfinp
    \label{eq:epi_convergence}
  \end{equation}
\end{proposition}

\begin{proof}
  We know by Theorem~\ref{th:lower_smooth_properties} that,
  for any $n \in \NN$, the functions
  $\na{\RawSmoothValueFunction^{\mu_n}_{t}}_{t\in\ic{0,\horizon}}$
  are lsc. Moreover, for each~$t \in \ic{0,\horizon}$,
  $\na{\RawSmoothValueFunction^{\mu_n}_{t}}_{n\in\NN}$ is a nondecreasing
  sequence of functions since $\na{\mu_n}_{n\in\NN}$
  is a nonincreasing sequence of positive real numbers. This ensures
  the equivalence between pointwise convergence and epiconvergence
  for the sequence $\na{\RawSmoothValueFunction^{\mu_n}_{t}}_{n\in\NN}$
  by~\cite[Proposition~7.4(d)]{rockafellar2009variational}, so that it is
  sufficient to prove the pointwise convergence of the sequence
  $\na{\RawSmoothValueFunction^{\mu_n}_{t}}_{n\in\NN}$
  to obtain~\eqref{eq:epi_convergence}:
  \begin{equation}
    \RawSmoothValueFunction^{\mu_n}_{t}
    \xrightarrow[n\to+\infty]{\textsc{p}} \ValueFunction_{t}
    \eqfinv
    \label{eq:pointwise_convergence}
  \end{equation}

  We proceed by backward induction.

  $\bullet$ We start by proving~\eqref{eq:pointwise_convergence} at
  stage~$\horizon$.  From~\eqref{eq:smooth_final},
  $\RawSmoothValueFunction^{\mu_{n}}_\horizon = \Final^{\mu_n}$ for all
  $n \in \NN$, where for $\state \in \STATE$, $\Final^{\mu_n}(\state, \cdot)$ is defined
  in~\eqref{eq:my_final_cost} as the Moreau envelope of $\Final(\state, \cdot)$.  By
  Assumption~\ref{as:convexity}, the function~$\Final$ is proper and belongs to
  $\Gamma[\STATE, \PARAM]$.  We consider two cases.  Either
  $\Final(\state, \cdot) = +\infty$ , in which case
  $\RawSmoothValueFunction^{\mu_{n}}_\horizon(\state \sep \cdot) =
  \Final^{\mu_n}(\state, \cdot) = +\infty$ for all $n \in \NN$ and obviously pointwise
  converges to $\Final(\state, \cdot)$.  Or $\Final(\state, \cdot)$ is proper, in which
  case
  $\RawSmoothValueFunction^{\mu_{n}}_\horizon(\state \sep \cdot) =
  \Final^{\mu_n}(\state, \cdot)$ converges pointwise to $\Final(\state, \cdot)$ and is a
  nondecreasing sequence, from the properties of the Moreau envelope (see
  \cite[Proposition~12.33]{Bauschke-Combettes:2017}).  This
  proves~\eqref{eq:pointwise_convergence} at final stage~$\horizon$.

  $\bullet$ Now, fix $t \in \ic{0, \horizon-1}$, and assume that
  the statement~\eqref{eq:pointwise_convergence} holds true at stage $t{+}1$.
  For any $n \in \NN$ and for any
  $\np{\state, \control, \param}
  \in \STATE \times \CONTROL \times \PARAM$,
  we define the lower smooth $Q$-function~$\RawSmoothQFunction_t^{\mu_n}$ by
  \begin{equation*}\label{eq:proof_QFunction_smooth}
    \RawSmoothQFunction_t^{\mu_n}(\state, \control \sep \param)
    =
    \EE
    \Bc{\StageCost_{t}^{\mu_n}(\state, \control, \rvUncertain_{t+1}, \param)
      + \RawSmoothValueFunction_{t+1}^{\mu_n}
      \bp{\dynamics_t\np{\state, \control, \rvUncertain_{t+1}} \sep \param}}
    \eqfinp
  \end{equation*}

  As a first step, we show that the sequence of functions
  $\na{\RawSmoothQFunction^{\mu_n}_t(\state, \cdot \sep \param)}_{n\in\NN}$
  epiconverges to $\QFunction_{t}(\state, \cdot \sep \param)$,
  with $\QFunction_{t}$ defined in~\eqref{eq:Q_function_expression}.
  From the properties of the Moreau envelope,
  the sequence $\na{\StageCost_t^{\mu_n}}_{n\in\NN}$ in~\eqref{eq:my_stage_cost}
  is nondecreasing and, from Lemma~\ref{le:smooth_cost},
  for any $\uncertain \in \Support{\rvUncertain_{t+1}}$,
  each function
  $\StageCost_t^{\mu_n}(\cdot, \cdot, \uncertain, \cdot)$ is lsc.
  Similarly, by assumption,
  the sequence $\na{\RawSmoothValueFunction_{t+1}^{\mu_n}}_{n\in\NN}$
  is nondecreasing and, from Theorem~\ref{th:lower_smooth_properties},
  each function $\RawSmoothValueFunction_{t+1}^{\mu_n}$ is lsc.
  It follows that
  $\na{\RawSmoothQFunction^{\mu_n}_t(\state, \cdot \sep \param)}_{n\in\NN}$
  is a nondecreasing sequence and that each function
  $\RawSmoothQFunction^{\mu_n}_t$ is lsc, since the expectation above
  is a finite sum.
  Therefore, from \cite[Proposition~7.4(d)]{rockafellar2009variational},
  we obtain the epiconvergence
  \begin{equation*}
    \RawSmoothQFunction^{\mu_n}_t(\state, \cdot \sep \param)
    \xrightarrow[n\to+\infty]{\textnormal{e}}
    \sup_{n \in \NN} \bp{\text{lsc} \bp{\RawSmoothQFunction^{\mu_n}_t}(\state, \cdot \sep \param)}
    =\sup_{n \in \NN} \bp{\RawSmoothQFunction^{\mu_n}_t(\state, \cdot \sep \param)}
    \eqfinp
  \end{equation*}
  Moreover, the nondecreasing sequences
  $\na{\StageCost_t^{\mu_n}}_{n\in\NN}$
  and $\na{\RawSmoothValueFunction_{t+1}^{\mu_n}}_{n\in\NN}$ of functions
  converge pointwise respectively to
  the function~$\StageCost_{t}$,
  using the properties of the
  Moreau envelope (see \cite[Proposition~12.33(ii)]{Bauschke-Combettes:2017}),
  and to $\ValueFunction_{t+1}$, using the backward induction assumption on
  $\na{\RawSmoothValueFunction_{t+1}^{\mu_n}}_{n\in\NN}$.
  It follows that
  $\sup_{n \in \NN}
  \bp{\RawSmoothQFunction^{\mu_n}_t(\state, \cdot \sep \param)}
  = \QFunction_{t}(\state, \cdot \sep \param)$.	

  As a second step, we show that~\eqref{eq:pointwise_convergence} holds at stage $t$.
  For that purpose we fix $(x,p)\in \STATE{\times}\PARAM$ and consider the sequence of mappings
  $\na{\RawSmoothQFunction^{\mu_n}_t(\state, \cdot \sep \param)}_{n\in\NN}$.
  Since $\domain \np{\RawSmoothQFunction^{\mu_n}_t(\state, \cdot \sep \param)}
  \subseteq \cap_{\uncertain\in\Support{\rvUncertain_{t+1}}}
  \domain \coutint_t^{\mu_n}(\state, \cdot, \uncertain, \param)$
  ---
  where $\domain \coutint_t^{\mu_n}(\state, \cdot, \uncertain, \param)$ is bounded
  as $\coutint_t^{\mu_n}(\cdot, \cdot, \uncertain, \cdot) \in \Gamma_\mathcal{K}[\STATE\times\CONTROL, \PARAM]$
  ---
  the sequence $\na{\RawSmoothQFunction^{\mu_n}_t(\state, \cdot \sep \param)}_{n\in\NN}$ of functions is
  \textit{eventually level-bounded}
  (see \cite[Exercice~7.32(a)]{rockafellar2009variational}).
  Therefore, by application of
  \cite[Theorem~7.33]{rockafellar2009variational},
  we obtain the convergence of the infima, that is,
  \begin{equation*}
    \inf_{\control \in \CONTROL}
    \RawSmoothQFunction^{\mu_n}_t(\state, \control \sep \param)
    \xrightarrow[n \to +\infty]{}
    \inf_{\control \in \CONTROL}
    \QFunction_t(\state, \control \sep \param)
    \eqfinp
  \end{equation*}
  Thus, we obtain the pointwise convergence of the sequence
  $\na{\RawSmoothValueFunction_{t}^{\mu_n}}_{n\in \NN}$ as,
  for any fixed value $(x,p)\in \STATE\times\PARAM$, we have that
  \[
    \RawSmoothValueFunction_{t}^{\mu_n}(\state \sep \param)= \inf_{\control \in \CONTROL}
  \RawSmoothQFunction^{\mu_n}_t(\state, \control \sep \param)
  \xrightarrow[n \to +\infty]{}
  \inf_{\control \in \CONTROL}
  \QFunction_t(\state, \control \sep \param) =
  \ValueFunction_{t}(\state \sep \param)
  \eqfinp
  \]
  Moreover, for all $n \in \NN$, we have that $\RawSmoothValueFunction^{\mu_{n}}_t
  \leq \RawSmoothValueFunction^{\mu_{n+1}}_t$ since the sequence
  $\na{\RawSmoothQFunction^{\mu_n}_t(\state, \cdot \sep \param)}_{n\in\NN}$
  is nondecreasing.
  We conclude that the statement~\eqref{eq:pointwise_convergence}
  holds at time $t$. This ends the proof.
\end{proof}

As a consequence, we obtain the following corollary.

\begin{corollary}
\label{co:convergence_of_inf}
  Under the assumptions of Proposition~\ref{pr:convergence_of_smooth_vf},
  let $\state_{0}$ be the initial state in~\eqref{eq:constraints_begin},
  let the set $\ParamAd$ be compact,
  and let $\Phi\opt = \inf_{\param \in \ParamAd} \Phi(\param)$
  be the optimal value of
  Problem~\eqref{eq:parametric_multistage_problem}. Then,
  we have that
  \begin{equation}
    \inf_{\param \in \ParamAd} \RawSmoothValueFunction^{\mu_n}_{0}(\state_{0} \sep \param)
    \leq \Phi\opt \eqsepv
    \forall n \in \NN \eqsepv
    \text{and} \quad
    \inf_{\param \in \ParamAd} \RawSmoothValueFunction^{\mu_n}_{0}(\state_{0} \sep \param)
    \xrightarrow[n\to+\infty]{} \Phi\opt
    \eqfinp
  \end{equation}
\end{corollary}

\begin{proof}
  From Proposition~\ref{pr:convergence_of_smooth_vf}, the sequence
  $\na{\RawSmoothValueFunction_0^{\mu_n}}_{n\in\NN}$ of functions converges
  pointwise and epiconverges to $\ValueFunction_{0}$
  in~\eqref{eq:parametric_value_functions}.  Moreover, the
  function~$\indicator_{\na{\state_{0}}\times\ParamAd}$ is lsc.  It follows from
  \cite[Proposition~7.46]{rockafellar2009variational} that the sequence
  $\na{\RawSmoothValueFunction_0^{\mu_n} +
    \indicator_{\na{\state_{0}}\times\ParamAd}}_{n\in\NN}$ of functions epiconverges to
  $\ValueFunction_{0} + \indicator_{\na{\state_{0}}\times\ParamAd}$.  Then, as the
  effective domain
  $\domain \np{\RawSmoothValueFunction_0^{\mu_n} +
    \indicator_{\na{\state_{0}}\times\ParamAd}} \subseteq \na{\state_{0}}\times\ParamAd$ is
  bounded, the sequence
  $\na{\RawSmoothValueFunction_0^{\mu_n} +
    \indicator_{\na{\state_{0}}\times\ParamAd}}_{n\in\NN}$ is \textit{eventually
    level-bounded} (see \cite[Exercice~7.32(a)]{rockafellar2009variational}).
  Therefore, by application of \cite[Theorem~7.33]{rockafellar2009variational},
  we obtain the convergence of the infimum
  \begin{equation*}
    \inf_{\param \in \ParamAd}
    \RawSmoothValueFunction^{\mu_n}_{0}(\state_{0} \sep \param)
    \xrightarrow[n\to+\infty]{}
    \inf_{\param \in \ParamAd}
    \ValueFunction_0(\state_{0} \sep \param)
    =
    \Phi\opt \eqfinp
  \end{equation*}
  Then, for any $n \in \NN$, the inequality
  $\inf_{\param \in \ParamAd} \RawSmoothValueFunction^{\mu_n}_{0}(\state_{0} \sep
  \param) \leq \Phi\opt$ follows from
  $\RawSmoothValueFunction^{\mu_n}_{0} \leq \RawSmoothValueFunction^{\mu_{n+1}}_{0}$ as
  shown in the proof of Proposition~\ref{pr:convergence_of_smooth_vf}.
\end{proof}

\begin{remark}
  In Corollary~\ref{co:convergence_of_inf}, by application of
  \cite[Theorem~7.33]{rockafellar2009variational}, we also obtain that all
  accumulation points of a sequence $\na{\param_n}_{n\in \NN}$ satisfying
  $\param_n \in \argmin_{\param \in \Param}
  \RawSmoothValueFunction^{\mu_n}_{0}(\state_{0} \sep \param)$ for $n \in \NN$ is a
  solution of $\inf_{\param \in \ParamAd} \Phi(\param)$
  in~\eqref{eq:parametric_multistage_problem}.  In particular, if
  Problem~\eqref{eq:parametric_multistage_problem} has a unique solution
  $\param\opt \in \ParamAd$, then $\param_n \xrightarrow[n\to+\infty]{} \param\opt$.
\end{remark}

\section{A First-Order Optimization Method Based on Recursive Gradient Computation}
\label{sec:experimental_assessment_based_on_SDDP}

Solving the PMSOP~\eqref{eq:parametric_multistage_problem} amounts to compute an
optimal parameter
\begin{equation}
  \param^\sharp \in \argmin_{\param \in \ParamAd} \Phi(\param) \eqfinp
\end{equation}
First, in~\S\ref{sec:estimation_Phi}, we introduce a numerical assessment method
that we use to estimate accurately the value $\Phi(\param\opt)$ for any candidate
PMSOP solution $\param\opt \in \ParamAd$.  Second,
in~\S\ref{sec:first_order_methods}, we introduce
a 
first-order optimization method based on recursive gradient computation for
solving the PMSOP~\eqref{eq:parametric_multistage_problem}.  In this method,
gradients $\nabla\Phi$ are computed thanks to the backward
induction~\eqref{eq:backward_induction} established in
Theorem~\ref{th:smooth_convex_problem}.  We also discuss
in~\S\ref{sec:first_order_methods_SDDP} the case of an alternative SDDP-based
first-order method that we introduce for comparison purposes.

\subsection{Numerical Assessment Method}
\label{sec:estimation_Phi}

Given a candidate optimal solution $\param\opt \in \ParamAd$
for the PMSOP~\eqref{eq:parametric_multistage_problem},
we need to compute an accurate estimation of $\Phi(\param\opt)$
to assess the quality of $\param\opt$.
For this purpose, we choose to rely on the Stochastic Dual Dynamic Programming
algorithm (SDDP, \cite{pereira1991multi}), with the original state, as it
provides both a lower bound and an upper bound for
$\Phi(\param\opt) = \ValueFunction_0(\state_0 \sep \param\opt)$.

First, we compute polyhedral lower approximations
$\na{\underline{\ValueFunction}_t\np{\cdot \sep \param\opt}}_{t \in \ic{0,\horizon}}$
of the parametric value functions
$\na{\ValueFunction_t\np{\cdot \sep \param\opt}}_{t \in \ic{0,\horizon}}$ with the
SDDP algorithm.
In this application of SDDP, we stress that the state space
is $\STATE$ and that
the parameter is fixed to the value
$\param\opt$ --- hence left apart from the state.

Second,
we use the resulting policy
\begin{align}
  \label{eq:policy}%
  \pi_t(x; \param\opt)
  &\in
    \argmin_{\control \in \ConstraintSet_t(\state)}
    \EE
    \Bc{\StageCost_{t}
    (\state, \control, \rvUncertain_{t+1} \sep \param\opt)
    + \underline{\ValueFunction}_{t+1}
    \bp{\dynamics_t\np{\state, \control, \rvUncertain_{t+1}}; \param\opt}}
    \eqsepv
  \\
  & \hspace{6cm}
    \forall \state \in \STATE \eqsepv
    \forall t \in \ic{0, \horizon-1}
    \eqfinv \nonumber
\end{align}
to compute the expected simulation cost
\begin{equation}
  \label{eq:upper_V}%
  \overline{\ValueFunction}_{0}(\state_0 \sep \param\opt)
  =
  \EE
  \Bc{\sum_{t=0}^{\horizon-1}
    \StageCost_{t}
    \bp{\rvState_t, \pi_t(\rvState_t; \param\opt), \rvUncertain_{t+1} \sep \param\opt}
    + \Final(\rvState_\horizon \sep \param\opt)}
  \eqfinp
\end{equation}
Since SDDP provides polyhedral lower estimates of the (true) parametric value
functions
$\na{\ValueFunction_{t}\np{\cdot \sep \param\opt}}_{t \in \ic{0,\horizon}}$ defined
by~\eqref{eq:parametric_value_functions}, and since
$\na{\pi_t}_{t\in\ic{0,\horizon-1}}$ in~\eqref{eq:policy} is a suboptimal policy for
Problem~\eqref{eq:parametric_multistage_problem}, we have the inequality
\begin{equation}
  \underline{\ValueFunction}_{0}(\state_0 \sep \param\opt)
  \leq
  \Phi(\prof\opt)
  \leq
  \overline{\ValueFunction}_{0}(\state_0 \sep \param\opt) \eqfinp
  \label{eq:estimating_bounds}%
\end{equation}

In practice, we compute
$\underline{\ValueFunction}_{0}(\state_0 \sep \param\opt)$ by running
$\maxiter \in \NN$ forward-backward passes of the SDDP algorithm, and the
expectation for $\overline{\ValueFunction}_{0}(\state_0 \sep \param\opt)$
in~\eqref{eq:upper_V} is computed by Monte-Carlo simulation, generating
scenarios with the discrete probability distributions of the noise variables
$\sequence{\rvUncertain_t}{t \in \ic{1, \horizon}}$ in~\eqref{eq:noise_process}.

\subsection{A First-Order Optimization Method Based on Recursive Gradient Computation for Solving a PMSOP}
\label{sec:first_order_methods}

Our main algorithmic contribution 
is to introduce a gradient-based first-order optimization method for solving a
PMSOP.

The method is showcased in Algorithm~\ref{alg:first-order}.  Starting from an
initial parameter $\param^0 \in \ParamAd$, we run at most $N \in \NN$ iterative
steps (outer loop), where, at each step, we call a first-order oracle
$\param \mapsto \bp{\ValueFunction_0(\state_0 \sep \param), \nabla_\prof
  \ValueFunction_0(\state_0 \sep \param)}$ based on the backward
induction~\eqref{eq:backward_induction} established in
Theorem~\ref{th:smooth_convex_problem} (inner loop). The output of the oracle is
used so as to update the current parameter value $\param^i$ to $\param^{i+1}$
according to a first-order update rule \cite{beck2017first}.  When the algorithm
stops, a candidate solution $\param\opt \in \ParamAd$ is returned.

\begin{algorithm}[H]
  \textbf{input:} $\param^0 \in \ParamAd$ \\
  \For{$i \in \ic{0, N{-}1}$}{
    $\blacktriangleright$
    compute
    $\bp{\ValueFunction_0(\state_0 \sep \param^i),
      \nabla_\param \ValueFunction_0(\state_0 \sep \param^i)}$ based on the backward induction~\eqref{eq:backward_induction}
    \ $(\bigstar)$ \\
    $\blacktriangleright$ update
    $\param^i$ to $\param^{i+1}$ using a first-order update rule
    and the output of the oracle \\
    $\blacktriangleright$ check if a stopping
    criterion is satisfied
  }
  \textbf{output:} $\param\opt \in \ParamAd$ \\
  \caption{Gradient-based first-order optimization method for a PMSOP}
  \label{alg:first-order}%
\end{algorithm}

When the parametric differentiability Assumption~\ref{as:smoothness} is not
satisfied, we cannot apply Theorem~\ref{th:smooth_convex_problem}.  Instead, we
resort to the regularization process introduced in~\S\ref{sec:lower_smooth}: we
select a regularization coefficient $\mu > 0$ and call the first-order oracle
$\param \mapsto \bp{\SmoothValueFunction_0(\state_0 \sep \param), \nabla_\param
  \SmoothValueFunction_0(\state_0 \sep \param)}$ based on the backward
induction~\eqref{eq:backward_induction_Moreau} of
Theorem~\ref{th:lower_smooth_properties} during phase $(\bigstar)$ of
Algorithm~\ref{alg:first-order}.

\begin{remark}
  In Algorithm~\ref{alg:first-order}, we intentionally keep the choice of a
  first-order update rule open. In fact, any smooth convex optimization
  algorithm based on the gradient of the objective function $\Phi$ can be
  used. Note that this includes quasi-Newton algorithms.
\end{remark}

\subsection{An Alternative First-Order Method Based on SDDP with
  Extended State}
\label{sec:first_order_methods_SDDP}

As an alternative to the gradient-based first-order method outlined in
Algorithm~\ref{alg:first-order}, we propose to replace the first-order oracle of
phase $(\bigstar)$ by an oracle
$\param \mapsto \bp{\PolyhedralValueFunctionbis_0^{\maxiter}\bp{(\state_0, \param)}, q
  \in \partial_\prof \PolyhedralValueFunctionbis_{0}^{\maxiter}\bp{(\state_0, \param)}}$
built on SDDP with extended\footnote{%
  This is why, here, value functions are denoted by
  \( \ValueFunction_{t}\bp{(\state, \param)} \) and not by
  \( \ValueFunction_{t}\np{\state \sep \param} \).}  state
$\statebis = (\state, \param)$ as in~\eqref{eq:extended_dynamics}.  This oracle
relies on polyhedral lower approximation
$\PolyhedralValueFunctionbis_{0}^{\maxiter}$ of the parametric value function
$\ValueFunction_0$ in~\eqref{eq:parametric_value_functions} computed by running
$\maxiter$ forward-backward passes of SDDP.

We stress that, in that case, the state variable needs to be extended to
$\statebis = (\state, \param)$ together with a trivial stationary dynamics for
its parameter component~$\param$, as explained in~\eqref{eq:extended_dynamics}.
Indeed, with such a state extension, subgradients in the extended dual space
$\STATE \times \PARAM$ are computed during the backward passes of SDDP.  Thus, a
subgradient
$q \in \partial_\prof \PolyhedralValueFunctionbis_{0}^{\maxiter}\bp{(\state_0,
  \param^i)}$ can be computed in phase $(\bigstar)$, at each step
$i \in \ic{1,N}$ of the outer loop of Algorithm~\ref{alg:first-order}.  We recall
that this state extension is expected to affect the numerical performance of
SDDP: we refer to \cite{shapiro2011analysis} for further details on the SDDP
algorithm.

Lastly, we also notice that this SDDP-based method is built on an oracle which
returns subgradients (and not gradients as in the method we have developed
in~\S\ref{sec:first_order_methods}).  It follows that the choice of a
first-order update rule is restricted to nonsmooth methods \cite{beck2017first}.

\section{An Example in Day-Ahead Power Scheduling}
\label{sec:day_ahead_problem}

We now consider a numerical example for the purpose of
testing algorithms
for the numerical solution of a PMSOP~\eqref{eq:parametric_multistage_problem}.

First, in~\S\ref{sec:day_ahead_problem_statement},
we introduce a PMSOP that represents the
minimization of the expected intraday management cost
of a solar plant.
Second, in~\S\ref{sec:day_ahead_problem_results},
we showcase the numerical results obtained. 

\subsection{Problem Statement}
\label{sec:day_ahead_problem_statement}

We introduce a PMSOP as defined in~\eqref{eq:parametric_multistage_problem} for
the daily management of a solar plant.  Our example is inspired by the French
regulation for day-ahead power scheduling of renewable units in islanded and
overseas territories~\cite{CRE}.

The schematic organization of the solar plant is given in Figure~\ref{fig:unit}.
The main components of the power unit are a DC/AC inverter, a solar panel of
installed peak power $\peak \in \spRR$ (MW), and a lithium-ion battery
characterized by the coefficients $\battery$ referring respectively to the
battery's capacity (MWh), minimum load (MW), maximum load (MW), charge and
discharge efficiency coefficients.

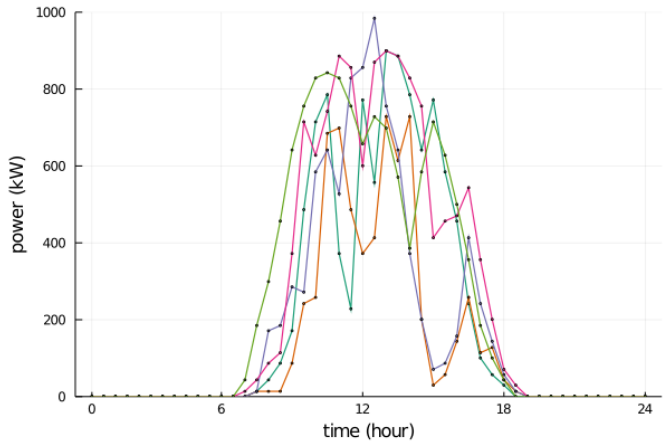
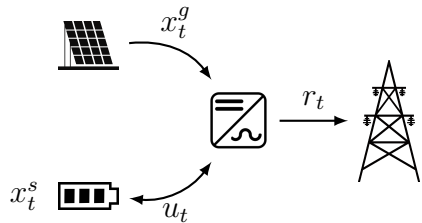
\begin{figure}[htpb]
  \begin{subfigure}[b]{0.55\linewidth}
    \centering
    \includegraphics[width=\linewidth]{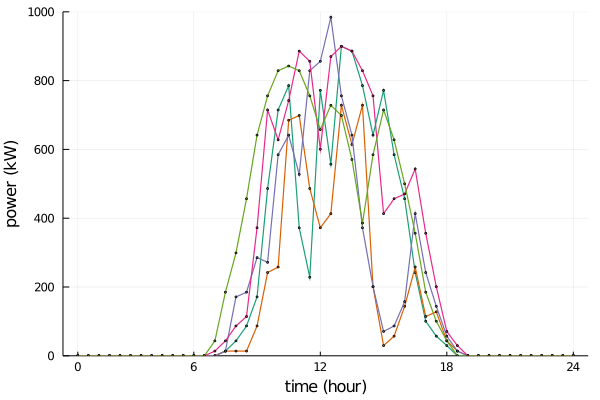}
    \caption{Example of daily generated power scenarios}
    \label{fig:ausgrid_data}
  \end{subfigure}
  \hfill
  \begin{subfigure}[b]{0.4\linewidth}
    \centering

    \begin{tikzpicture}

      \node (bus) at (0, 0) {\includegraphics[width=0.8cm]{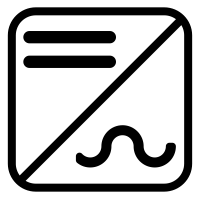}};
      \node (solar) at (-2, 1) {\includegraphics[width=0.8cm]{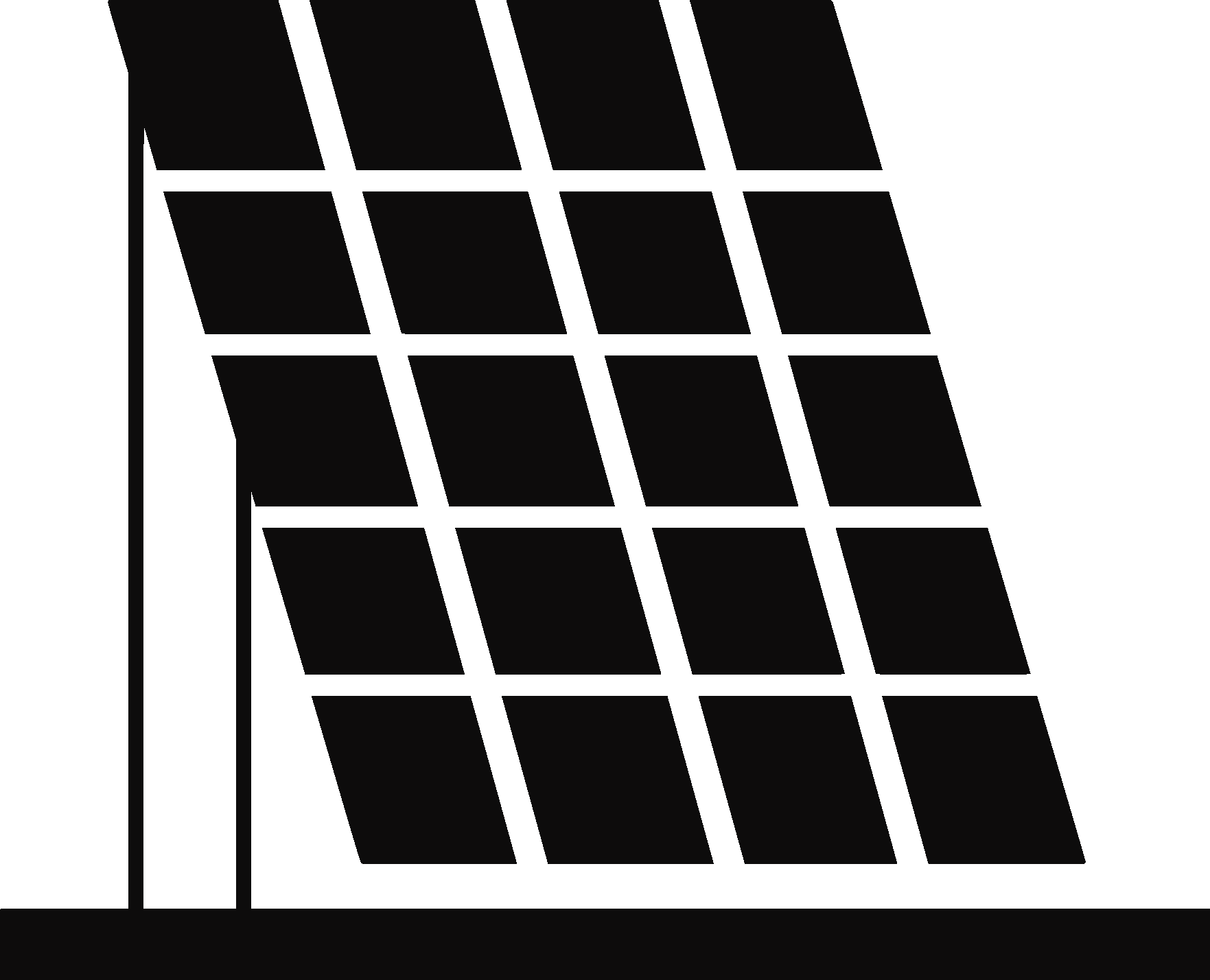}};
      \node (bat) at (-2, -1) {\includegraphics[width=0.8cm]{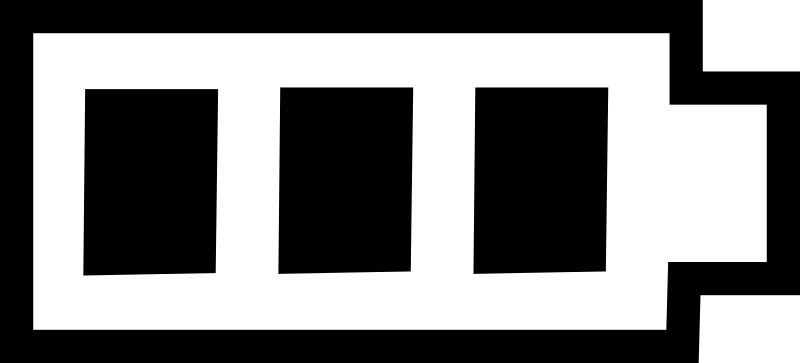}};
      \node (grid) at (2, 0) {\includegraphics[width=0.8cm]{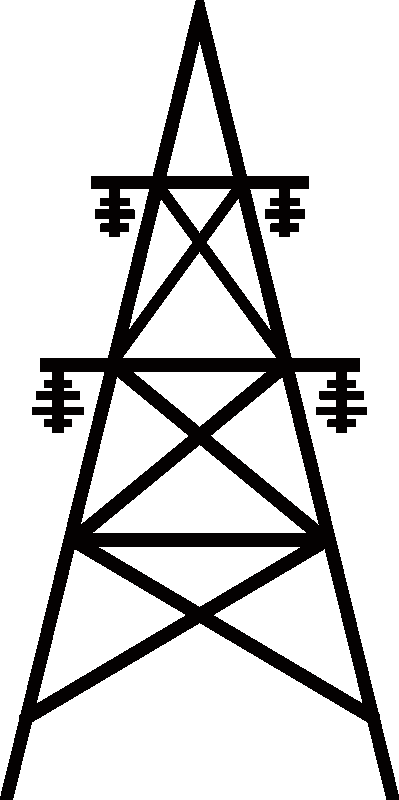}};

      \draw[->, >=latex, thick] (solar) to[bend left] node[above]{$\state^\gen_t$} (bus) ;
      \draw[<->, >=latex, thick] (bat) node[left=0.5cm]{$\state^\soc_t$} to[bend right] node[below]{$\control_t$} (bus) ;
      \draw[->, >=latex, thick] (bus) to node[above]{$\del_t$} (grid) ;

    \end{tikzpicture}

    \vspace{1cm}

    \caption{Schematic power plant}
    \label{fig:unit}
  \end{subfigure}
  \caption{Description of the solar power plant: (a) examples of
    daily generated power scenarios (data from~\cite{ausgrid})
    and (b) schematic organization}
  \label{fig:delta_gain}%
\end{figure}

We consider the time span of one operating day,
with time intervals of length $\Delta_t = 30$ minutes,
hence a problem horizon of $\horizon = 48$.
We now introduce all components of a dynamical system
to formalize our management problem.

\subsubsubsection{Variables and Parameter}

We introduce the state variable
\begin{align}
  \state_t
  &=
    \left(\begin{array}{l}
      \state_t^\soc \\[0.1cm]
      \state_t^\gen
    \end{array}\right)
    \in \RR^2 \eqsepv \forall t \in \ic{0, \horizon} \eqfinv
    \label{eq:state}
\end{align}
where
$\state_t^\soc \in \nc{0, 1}$ is the state of charge of the battery
and $\state_t^\gen \in \nc{0, \peak}$ is the generated power of the
solar panels,
both observed at stage $t \in \ic{0, \horizon}$.
The control
\begin{equation}
  \control_t \in \nc{\minBat, \maxBat} \eqsepv \forall t \in \ic{0, \horizon-1} \eqsepv
  \label{eq:control}
\end{equation}
taken at the beginning of every time interval~$[t,t+1[$,
accounts for the charging power ($\control_t \geq 0$)
or discharging power ($\control_t \leq 0$)
applied to the battery during~$[t,t+1[$.
Lastly, we introduce the noise variable
\begin{equation}
  \uncertain_t \in \RR \eqsepv \forall t \in \ic{1, \horizon} \eqfinv
  \label{eq:noise}%
\end{equation}
to represent uncertainties in the evolution
of the generated power $\state_t^\gen$.
All in all,
the two dimensional state~$\state_t$
evolves according to the dynamics\footnote{%
  The initial value $\state^\gen_0 = 0$ reflects the absence of sun at midnight
  in our use case.}
\begin{subequations}			
  \begin{equation}
    \state_{0} =
    \left(
      \begin{array}{l}
        \state_0^\soc \\ 0
      \end{array}\right) \eqsepv
    \state_{t+1} = \dynamics_t(\state_t, \control_t, \uncertain_{t+1}) \eqsepv \forall t \in \ic{0, \horizon-1} \eqfinv
  \end{equation}
  whose components are given by\footnote{%
Later, by means of constraints on the control, we will ensure that the dynamics~\eqref{eq:dynamics_bat}
below preserve the state constraint $\state_t^\soc \in \nc{0, 1}$
for all time~$t$.
 With the second part of the dynamics~$\dynamics_t$ in~\eqref{eq:dynamics_toy},
we cannot ensure that $\state_t^\gen \in \nc{0, \peak}$ for all time~$t$
because the uncertainty~$\uncertain$ is additive.
  In numerical practice, we round values $\state^\gen \notin \nc{0,\peak}$
  by projecting them to $\nc{0,\peak}$.}
  \begin{align}
    \dynamics_t(\state, \control, \uncertain)
    &= \left(\begin{array}{l}
               \dynamics(\state^\soc, \control) \\
               \alpha_t \state^\gen + \beta_t + \uncertain
             \end{array}\right) \eqsepv
    \forall \np{\state, \control, \uncertain} \in
    \RR^2 \times \RR \times \RR \eqsepv \forall t \in \ic{0, \horizon-1} \eqfinv
    \label{eq:dynamics_toy}%
    \\
    \dynamics\np{\state^\soc, \control}
    &=
      \state^\soc + \Bp{\rho_c \control^+ - \frac{1}{\rho_d} \control^-}
      \frac{\Delta_t}{\kappa}
      \eqsepv \forall \np{\state^\soc, \control} \in \nc{0, 1}\times \RR
      \eqfinv
      \label{eq:dynamics_bat}%
  \end{align}
\end{subequations}
where
\( \control^+ = \max(0, \control) \)
and \( \control^- = \max(0, -\control) \);
$\dynamics$ in~\eqref{eq:dynamics_bat} is the dynamics of the state of charge
of the battery; $\np{\alpha_t, \beta_t}$ in~\eqref{eq:dynamics_toy}
are the weights of the linear dynamics (AR(1) model) of the generated power
at stage $t \in \ic{0, \horizon-1}$.

We model the uncertainty of the noise variable $\uncertain_t$ in~\eqref{eq:noise}
with a stochastic noise process
$\rvUncertain = \na{\rvUncertain_t}_{t \in \ic{1,\horizon}}$
as in~\eqref{eq:noise_process},
that we assume
to be stagewise independent
with finitely supported random variables $\rvUncertain_t$,
$t \in \ic{1,\horizon}$ ---
in line with the discrete white noise
Assumption~\ref{as:discrete_white_noise}---
and also introduce
a state process $\rvState = \na{\rvState_t}_{t \in \ic{0,\horizon}}$ as in~\eqref{eq:rvState} and a control process
$\rvControl = \na{\rvControl_t}_{t \in \ic{0,\horizon-1}}$ as
in~\eqref{eq:rvControl}

Lastly, the solar plant is engaged \emph{from the start} to deliver,
for each time interval $[t, t+1[$,
a certain value of committed power $\prof_t \in \RR$,
composing altogether the parameter
\begin{equation}
  \prof = \na{\prof_t}_{t \in \ic{0, \horizon-1}}
  \in \RR^\horizon \eqfinp
  \label{eq:committed_profile}
\end{equation}
Despite the fact that the parameter~$p$ is itself a temporal sequence
$p= \na{p_t}_{t\in\ic{0,T-1}} $, we stress that $p$ is an \emph{initial
  decision} which is made prior to the observation of any uncertain
outcome. Indeed, it models the amount of power to be delivered for each time
step $t\in\ic{0,T-1}$, a decision which is made day-ahead --- that is, at time
$t=0$.  Thus, as opposed to the controls $\na{\control_t}_{t\in\ic{0,T-1}}$, the
initial decisions $\na{p_t}_{t\in\ic{0,T-1}}$ cannot be constructed as state ---
or history --- feedbacks.

\subsubsubsection{Constraints and Costs}

Controls are constrained by the admissibility sets
\begin{equation}
  \ConstraintSet_t(\state) = \defset{\control \in \RR }%
  {\minBat \leq \control \leq \maxBat
    \mtext{ and } 0 \leq \dynamics\np{\state^\soc, \control} \leq 1}
  \eqsepv \forall \state \in \RR^2 \eqsepv \forall t \in \ic{0, \horizon-1} \eqfinp
  \label{eq:toy_admissibility_set}%
\end{equation}
Besides, given the stagewise independence
assumption on noises
and the expression of the state variable
$\state_t$ in~\eqref{eq:state},
the nonanticipativity constraint in~\eqref{eq:constraints_end}
can be reformulated
without loss of optimality \cite[\S4.4]{Carpentier-Chancelier-Cohen-DeLara:2015} as
\begin{equation}
  \sigma(\rvControl_t) \subseteq \sigma(\rvState_t)
  \eqsepv \forall t \in \ic{0, \horizon-1}
  \eqfinp
\end{equation}

Stage costs depend on 
the delivered power~$\del_{t+1}$ over the interval~$[t, t+1[$,
given by
\begin{equation}\label{eq:delivered_power}
  \del_{t+1} = \state^\gen_{t+1} - \control_t
  = \alpha_t \state^\gen_t + \beta_t + \uncertain_{t+1} - \control_t
  \in \RR \eqsepv
  \forall t \in \ic{0, \horizon-1} \eqfinp
\end{equation}
Thus, for
$t \in \ic{0, \horizon-1}$
and
$\np{\state, \control, \uncertain, \prof_t}
\in \RR^2\times \RR \times \RR \times \RR$,
we define stage costs as
\begin{subequations}
  \label{eq:overall_cost}%
  \begin{equation}
    \StageCost_t(\state, \control, \uncertain \sep \param)
    =
    \Energy_t(\alpha_t \state^\gen + \beta_t + \uncertain - \control) +
    \Penalty_t(\alpha_t \state^\gen + \beta_t + \uncertain - \control, \prof_t)
    \eqfinv \label{eq:toy_stage_cost}
  \end{equation}
  with the energy cost $\Energy_t$ and the penalty cost $\Penalty_t$
  given by
  \begin{align}
    \Energy_t(\del_{t+1})	
    &= -c_t \Delta_t \del_{t+1} \eqfinv
    \\
    \Penalty_t(\del_{t+1}, \prof_t)
    &= \lambda c_t \Delta_t \abs{\del_{t+1} - \prof_t} \eqfinv \label{eq:toy_penalty}
  \end{align}
\end{subequations}
where $c_t$ is the (deterministic) energy price, expressed in \euro/MWh,
for $t \in \ic{0, \horizon}$ and $\lambda \geq 1$ is a penalty coefficient.
The final cost is then defined as
\begin{equation}
  \Final(\state) = -c_\horizon \state^\soc_\horizon \kappa \eqsepv
  \forall \state \in \RR^2 \eqfinp
  \label{eq:final_cost}%
\end{equation}
The cost structure defined by~\eqref{eq:overall_cost}-\eqref{eq:final_cost}
reflects the original formulation in~\cite{CRE},
except for the penalty cost $\Penalty_t$ in~\eqref{eq:toy_penalty}
that we have simplified for this illustrative example.

\subsubsubsection{Optimization Problem}
Gathering all components introduced above,
we define
the \emph{intraday value function} $\Phi$
as the value of the PMSOP
defined by~\eqref{eq:parametric_multistage_problem}.
Finally, we introduce the \emph{set of admissible parameter profiles}
\begin{equation}
  \ParamAd = \nc{0, \peak}^\horizon \eqfinv
  \label{eq:toy_constraint_set}
\end{equation}
and we consider the day-ahead optimization of the
expected management cost in~\eqref{eq:toy_stage_cost}
with respect to the commitment profile $\prof \in \ParamAd$,
that is, we want to solve
\begin{equation}
  \min_{\prof \in \ParamAd} \Phi(\prof) \eqfinp
  \label{eq:toy_problem}
\end{equation}

\subsection{Numerical Experiments}
\label{sec:day_ahead_problem_results}

We perform numerical experiments with
a single computer equipped with 4 Intel Core i7-7700K CPU
and 15 GB of RAM.
We use the package \texttt{SDDP.jl} \cite{dowson_sddp.jl}
for SDDP
together with the LP solver of CPLEX 12.9.
Apart from the solver, all our code is
implemented with the Julia language \cite{bezanson2012julia}.
Further implementation details are provided in Appendix~\ref{app:detailed_setup}.

First, we describe our experimental protocol,
where we introduce
two candidate methods to address Problem~\eqref{eq:toy_problem}.
Second, we comment on the results of the
two methods separately,
before finally confronting the results of both methods.

\subsubsubsection{Protocol for Numerical Experiments}

Our goal is to evaluate optimization methods for solving
Problem~\eqref{eq:toy_problem}.
The objective function $\Phi = \ValueFunction_0(\state_0 \sep \cdot)$
is convex (see Theorem~\ref{th:smooth_convex_problem})
and polyhedral
(due to the polyhedral penalty $\Penalty_t$ in~\eqref{eq:toy_penalty}
and other affine constraints,
see e.g. arguments in~\cite[\S3.2.1]{Shapiro-Dentcheva-Ruszczynski:2014}),
hence nondifferentiable.
We consider two methods.
\begin{itemize}
\item \textbf{$\boldsymbol{\mu}$SDP}:
  we address Problem~\eqref{eq:toy_problem}
  by solving
  \begin{equation}
    \min_{\prof \in \ParamAd} \ \SmoothValueFunction_0(\state_0 \sep \prof) \eqfinv
    \label{eq:toy_smooth_problem}
  \end{equation}
  following the gradient-based first-order optimization method introduced
  in~\S\ref{sec:first_order_methods}. In particular, as
  Problem~\eqref{eq:toy_problem} is not smooth, we resort to the alternative of
  Algorithm~\ref{alg:first-order} based on the lower smooth value function
  $\na{\SmoothValueFunction_t}_{\ic{0, \horizon}}$ defined
  by~\eqref{eq:smooth_value_functions} and compute gradients
  $\nabla_\prof \SmoothValueFunction_0(\state_0 \sep \cdot)$ using the Bellman-like
  recursion~\eqref{eq:backward_induction_Moreau} of
  Theorem~\ref{th:lower_smooth_properties}.

\item \textbf{$\boldsymbol{k}$SDDP}:
  we address Problem~\eqref{eq:toy_problem}
  by solving
  \begin{equation}
    \min_{\prof \in \ParamAd} \ \PolyhedralValueFunctionbis_0^k
    \bp{(\state_0, \prof)}
    \eqfinv
    \label{eq:toy_polyhedral_problem}
  \end{equation}
  following the the SDDP-based first-order optimization method introduced
  in~\S\ref{sec:first_order_methods_SDDP}.  We recall that this method is also
  an alternative to Algorithm~\ref{alg:first-order}, where we use a first-order
  oracle based on SDDP with the extended state space $\STATE \times \PARAM$ for phase
  $(\bigstar)$ of Algorithm~\ref{alg:first-order} in~\S\ref{sec:first_order_methods}.
\end{itemize}

For both methods, we consider several implementations,
playing on the coefficients that
affect the precision in the approximation of the objective function~$\Phi$
and the computing time.
\begin{itemize}
\item For $\mu$SDP, we have 12 implementations, characterized
  by the size of the discrete grids for state and control
  variables. We keep the regularization coefficient~$\mu$
  fixed and equal to~$0.1$ for all 12 implementations.
\item For $k$SDDP, we have 9 implementations, characterized by
  the number of forward-backward passes $k\in\NN$ run by
  the SDDP algorithm.
\end{itemize}

Each implementation of a method is processed as follows, following the
  approach developed in Sect.~\ref{sec:experimental_assessment_based_on_SDDP}.

$(i)$
First, we compute a commitment profile $\prof\opt \in \ParamAd$
as a solution of~\eqref{eq:toy_smooth_problem}
or~\eqref{eq:toy_polyhedral_problem}.
We initialize $\prof_0 = 0 \in \RR^\horizon$
and use a dynamic step size of $\eta_i = 10^3/i$ 
for both methods~$\mu$SDP and~$k$SDDP.
Note that~\eqref{eq:toy_smooth_problem}
is a smooth problem, and we could expect better performance with line
search. Yet, we keep the same step size for both methods
to focus the experiment on the performance of the oracles.
As for our stopping rule,
we stop the computation if we exceed 100 iterations,
or if the progress of the objective value is not larger than $\pm$0.5\%
for 5 consecutive iterations.
We report time performance for computing $\param\opt$,
both in term of overall computing time
(Figure~\ref{fig:overall_time}, $X$-axis in log scale, the lower the better)
and in term of average time per oracle call
(Figure~\ref{fig:oracle_time}, idem).

$(ii)$ Second, we compute an estimation of the optimal value~$\Phi(\param\opt)$
with SDDP, following the method introduced in~\S\ref{sec:estimation_Phi}.  In
Figure~\ref{fig:overall_time} and Figure~\ref{fig:oracle_time}, we report cost
performance ($Y$-axis, the lower the better), where the height of a marker spans
over the interval
$\nc{\underline{\ValueFunction}_{0}(\state_0; \param\opt),
  \overline{\ValueFunction}_{0}(\state_0; \param\opt)}$ on the $Y$-axis,
representing an estimation of the intraday value $\Phi(\prof\opt)$, from the
inequality~\eqref{eq:estimating_bounds}.  We report that, for each
implementation, the gap between the lower bound
$\underline{\ValueFunction}_{0}(\state_0; \param\opt)$ --- computed by running
2,000 forward-backward passes of SDDP--- and the expected simulation cost
$\overline{\ValueFunction}_{0}(\state_0; \param\opt)$ in~\eqref{eq:upper_V} ---
computed by Monte-Carlo simulation, sampling 25,000 scenarios --- is lower than
1.7\%.

Further detailed results are available in~Appendix~\ref{app:detailed_results}.

\begin{figure}[htpb]
  \centering
  \includegraphics[width=0.8\textwidth]{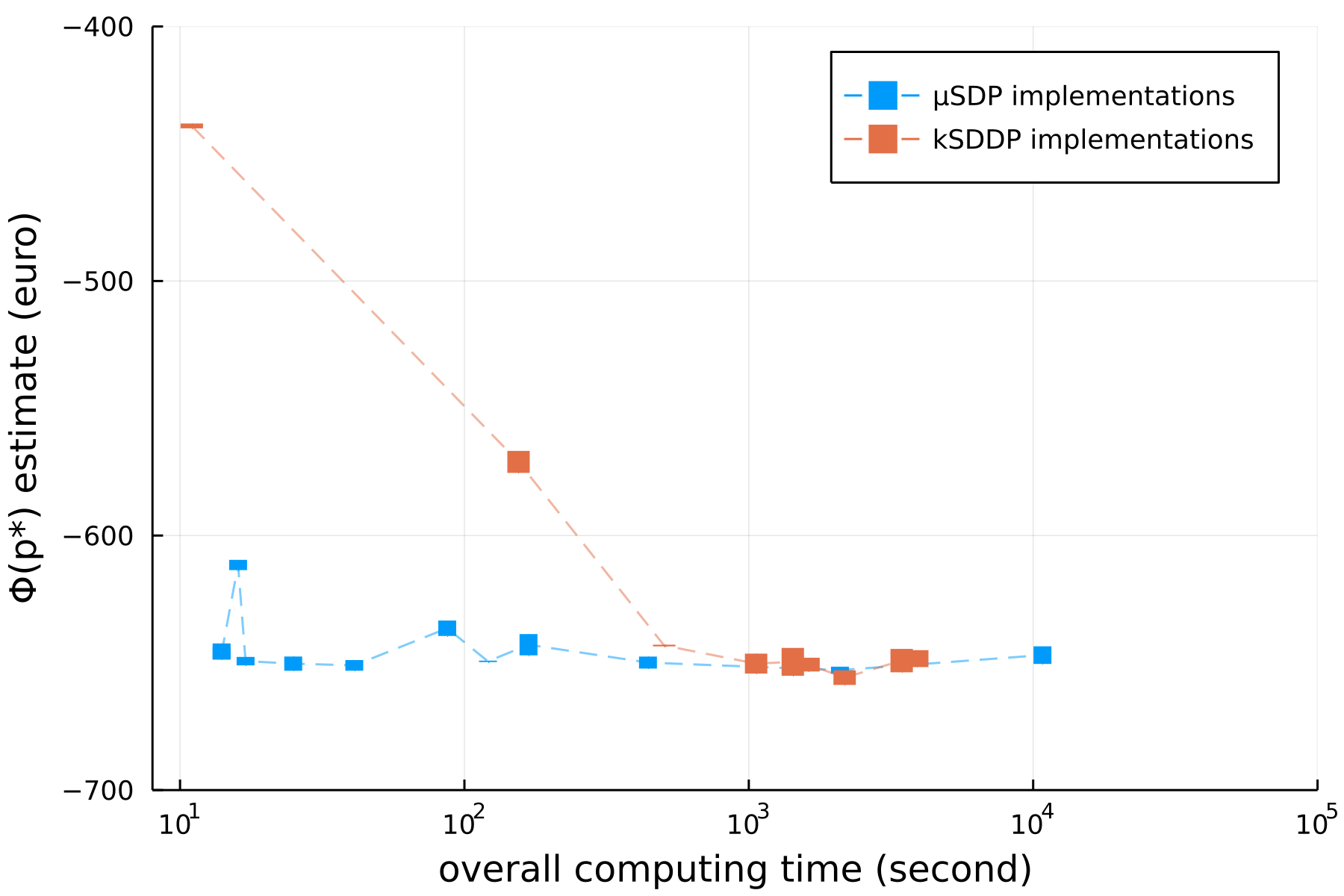}
  \captionsetup{width=\linewidth}
  \caption{
    Estimate of $\Phi(\prof\opt) \in \nc{\underline{\ValueFunction}_{0}(\state_0; \param\opt), \overline{\ValueFunction}_{0}(\state_0; \param\opt)}$
    in~\eqref{eq:estimating_bounds}
    (marker span on the $Y$-axis)
    for implementations of
    $\mu$SDP (in blue) and
    $k$SDDP (in orange);
    and overall computing time
    for the computation of $\prof\opt$
    ($X$-axis in log scale). For both axis, the lower the values the better.}
  \label{fig:overall_time}
\end{figure}

\begin{figure}[htpb]
  \centering
  \includegraphics[width=0.8\textwidth]{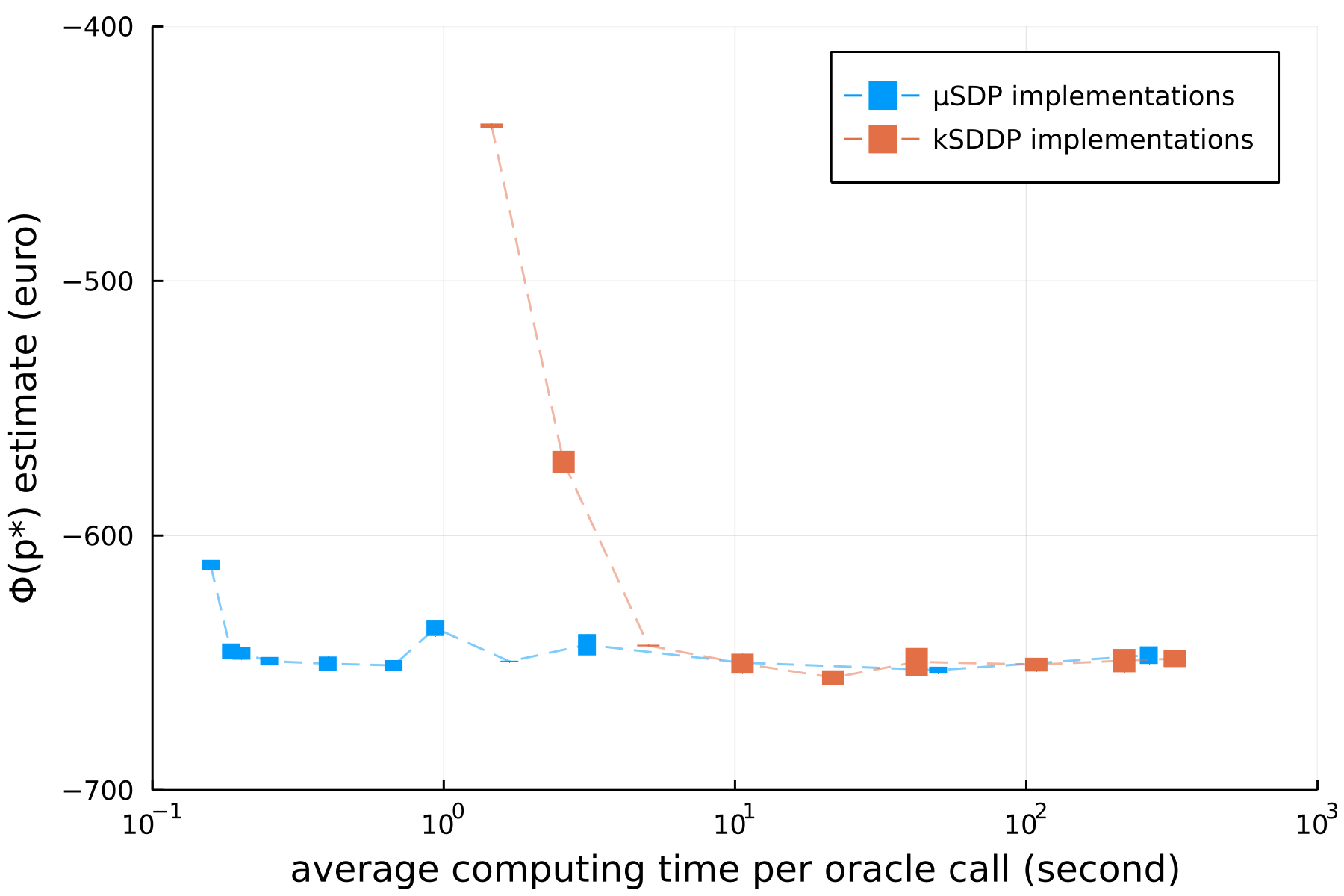}
  \captionsetup{width=\linewidth}
  \caption{
    Estimate of $\Phi(\prof\opt) \in \nc{\underline{\ValueFunction}_{0}(\state_0; \param\opt), \overline{\ValueFunction}_{0}(\state_0; \param\opt)}$
    in~\eqref{eq:estimating_bounds}
    (marker span on the $Y$-axis)
    for implementations of
    $\mu$SDP (in blue) and
    $k$SDDP (in orange);
    and average computing time per oracle call
    for the computation of $\prof\opt$
    ($X$-axis in log scale). For both axis, the lower the values the better.
  }
  \label{fig:oracle_time}
\end{figure}

\subsubsubsection{Results of $\boldsymbol{\mu}$SDP}

We comment on the results of implementations of $\mu$SDP,
represented by blue markers
on Figure~\ref{fig:overall_time} and Figure~\ref{fig:oracle_time}.
The performance for this method is related to
the size of the discrete grid introduced for
the state and control variables $(\state, \control)$ in~\eqref{eq:state}
and~\eqref{eq:control}.
In this experiment, we use 12 grid sizes
ranging from (5$\times$5, 11) points to (101$\times$101, 201) points.

We expect that
the finer the grid discretization, the more accurate the computation
of the value function $\SmoothValueFunction_{0}$ and of the gradient
$\nabla_\prof\SmoothValueFunction_{0}$, but also the longer the computing
time per oracle call.
Indeed, the worst cost performance
(highest estimated value of $\Phi(\prof\opt)$ on the $Y$-axis
of Figure~\ref{fig:overall_time} and Figure~\ref{fig:oracle_time})
is obtained with the coarsest grid discretization.
Although this implementation has the shortest average oracle time
(lowest value on the $X$-axis of Figure~\ref{fig:oracle_time}),
it requires more iterations of the projected gradient method
to stabilize, and is not the fastest implementation
in term of overall computing time
(second lowest value on the $X$-axis of Figure~\ref{fig:overall_time}).
More surprisingly, we find that the cost performance is greatly
improved by only adding one grid point to the state space
of the worst cost-performing implementation
(marker with the second lowest value on the $X$-axis of Figure~\ref{fig:oracle_time}),
and that cost performance stabilizes for further refinements of the grid
despite longer average oracle times
(third to twelfth markers on the same axis).

\subsubsubsection{Results of $\boldsymbol{k}$SDDP}

We comment on the results of implementations of $k$SDDP,
represented by orange markers
on Figure~\ref{fig:overall_time} and Figure~\ref{fig:oracle_time}.
The performance of these implementations is related to the
number $k \in \NN^*$ of forward-backward passes performed
by the SDDP algorithm at each iteration of the subgradient method.
We report results for a total of 9 implementations, with
$k \in \na{10, 20, 40, 80, 150, 250, 500, 750, 1000}$.

As expected, we observe that
the more forward-backward passes of the SDDP algorithm we run,
the more accurate the approximations of the value function
$\ValueFunction_{0}$ by $\PolyhedralValueFunctionbis_{0}^k$
and of the subdifferential $\partial_\prof\ValueFunction_{0}$
by $\partial_\prof\PolyhedralValueFunctionbis_{0}^k$.
This phenomenon is well illustrated by
Figure~\ref{fig:oracle_time},
where better cost performances,
corresponding to lower values on the $Y$-axis,
are obtained with longer computing
time per oracle call,
characterized by higher values on the $X$-axis
(where the higher the value of $k$,
the higher the position of the marker for implementations of $k$SDDP).
In general, we also report that low values of $k$
increase the noise of the oracle.
Indeed, for implementations
of $k$SDDP with $k \in\na{40, 80, 150}$,
the values of $\PolyhedralValueFunctionbis_{0}^k
\bp{(\state_0, \cdot)}$
do not stabilize after 100 iterations of the subgradient method,
whereas
implementations with $k \in \na{250, 500, 750, 1000}$
reach convergence in respectively $\na{34, 15, 16, 12}$ iterations.
As for $k \in\na{10, 20}$,
the SDDP algorithm only samples
a few scenarios in these cases,
and seems to fail
to obtain accurate representations of the
value function $\ValueFunction_{0}$.
Observations on oracle stability are backed by
more detailed results
available in Appendix~\ref{app:detailed_results}.

\subsubsubsection{Cross-Method Comparison}

We compare the results obtained with $\mu$SDP versus
the ones obtained with $k$SDDP.
We observe in Figures~\ref{fig:overall_time}
and~\ref{fig:oracle_time}
that $\mu$SDP (blue markers) almost attains its best cost performance
with a value of the expected simulation cost
$\overline{\ValueFunction}_{0}(\state_0 \sep \param\opt)$ in~\eqref{eq:upper_V}
of -648~\euro\ (upper value of the span of the markers
on the $Y$-axis) with only 0.25 seconds spent per oracle call
(Figure~\ref{fig:oracle_time}, $X$-axis)
and only 17 seconds of overall computing time
(Figure~\ref{fig:overall_time}, $X$-axis).
Comparatively,
for $k$SDDP implementations (orange markers), we need to
perform at least $k=80$ forward-backward passes of the SDDP algorithm
to attain $\overline{\ValueFunction}_{0}(\state_0 \sep \param\opt)$= -646 \euro\
(upper value of the span of the markers on the $Y$-axis).
For this value of $k$, we spend on average 10.2 seconds
per oracle call (Figure~\ref{fig:oracle_time}, $X$-axis),
and the overall computing time
is of 1,061 seconds (Figure~\ref{fig:overall_time}, $X$-axis).
We conclude that the $\mu$SDP method
performs better than $k$SDDP in our experiments, both in term
of time performance and in term of cost performance.
In particular, we argue that this result
illustrates the pertinence of treating the parameter
$\param$ in~\eqref{eq:parameter} apart from the state variables
to address Problem~\eqref{eq:parametric_multistage_problem},
especially when the parameter space $\PARAM$
is of much larger dimension than the state space $\STATE$.

\section{Conclusion}

We have studied differentiability properties (with respect to a parameter) of
a class of parametric multistage stochastic optimization problems.
Our main finding is that, under differentiability
and convexity assumptions,
we manage to compute, by backward induction, the gradient of the value function of the problem
with respect to the parameter.

In the case where the differentiability assumption is not fulfilled,
we have proposed a method for obtaining lower approximations
of the original parametric value functions by regularization.
We have also shown that such approximations
let us approach the value and a solution of the
original problem as closely as desired --- provided the
regularization coefficient is sufficiently small.

A numerical test case inspired from power scheduling
reveals that
our gradient computation technique
is efficient to formulate first-order oracles
in convex parametric multistage stochastic optimization.

In particular, we report that treating the parameter
apart from the state variables --- as proposed in our
backward induction for gradients ---
can be
advantageous in term of computing time performance,
especially when the problem is structured with
a large parameter space.

\appendix

\section{Appendix to Numerical Experiments}

\subsection{Data and Implementation Details}
\label{app:detailed_setup}

First, we detail the implementation of
the probabilistic model introduced
in~\S\ref{sec:day_ahead_problem_statement}.
We use one year of historical photovoltaic power
data from the public platform of
the Australian transmission
system operator Ausgrid \cite{ausgrid}.
We scale the generated power data to simulate the
operating of a solar power plant with an installed peak power
$\peak = 1$ MW.
Then, we use a standard linear regression to
calibrate the weights of the linear model
in~\eqref{eq:dynamics_toy},
and we perform a quantization of the
support of the error noise process $\sequence{\rvUncertain_t}{t \in \ic{1, \horizon}}$ with the $K$-means algorithm.
This latter technique lets us compute
discrete probability distributions for
each random variable in
$\sequence{\rvUncertain_t}{t \in \ic{1, \horizon}}$.
We refer the reader to~\cite{rujeerapaiboon2018scenario}
for the theoretical motivations of this quantization scheme.

Second, we provide implementation details for the oracles
of the two methods introduced in~\S\ref{sec:day_ahead_problem_results}.
For the oracle of the $\mu$SDP method,
we implement the backward recursions
in~\eqref{eq:smooth_value_functions} (for value functions)
and in~\eqref{eq:backward_induction} (for gradients)
by parallelizing computations across a discrete grid of states.
We also use a discrete grid for controls.
As for the oracle of the $k$SDDP method,
we use the built-in parallelization
scheme of \texttt{SDDP.jl} to run forward-backward passes
in asynchronous mode.

Lastly, concerning other parameters,
we take $\kappa = 1$ MWh, $\maxBat = -\minBat = 1$ MW,
and $\rho_c = \rho_d = 0.95$
for the battery parameters;
$c_t = 0.4$ \euro\
for the off-peak energy price
and $c_t = 0.6$ \euro\
for the on-peak energy price,
with a 2-hours peak spanning over [19:00, 21:00];
and $\lambda = 2$ for the penalty cost
in~\eqref{eq:toy_penalty}.

\subsection{Detailed Results for All Implementations}
\label{app:detailed_results}

We provide additional details on the numerical results
of each implementations
of the $\mu$SDP and $k$SDDP methods
considered in the experiments of~\S\ref{sec:day_ahead_problem_results}.

\begin{table}[H]
  \centering
  \begin{tabular}{l|rrrrrr}
    \begin{tabular}{@{}c@{}}
      $\boldsymbol{\np{\state, \control}}$ \\ \textbf{grid size} \end{tabular}
    & \begin{tabular}{@{}c@{}}
        \textbf{Iterative} \\ \textbf{steps}\end{tabular}
    & \begin{tabular}{@{}c@{}c@{}}
        \textbf{Overall} \\ \textbf{time} \\
        (seconds)\end{tabular}
    & \begin{tabular}{@{}c@{}c@{}}
        \textbf{Avg. time /} \\ \textbf{oracle call} \\ (seconds)\end{tabular}
    & \begin{tabular}{@{}c@{}c@{}}
        $\boldsymbol{\underline{\ValueFunction}_{0}(\state_0 \sep \param\opt)}$ \\
        \textbf{in \eqref{eq:estimating_bounds}} \\ (\euro)\end{tabular}
    & \begin{tabular}{@{}c@{}c@{}}
        $\boldsymbol{\overline{\ValueFunction}_{0}(\state_0 \sep \param\opt)}$ \\
        \textbf{in \eqref{eq:estimating_bounds}} \\ (\euro)\end{tabular}
    & \begin{tabular}{@{}c@{}}
        \textbf{Gap} \\ (\%)\end{tabular}
    \\
    \hline
    5$\times$5, 11
    & 97 & 16 & 0.16 & -613.6 & -609.6 & 0.7 \\
    5$\times$6, 11
    & 78 & 14 & 0.19 & -648.5 & -642.4 & 0.9 \\
    6$\times$6, 11
    & 69 & 14 & 0.20 & -648.8 & -643.6 & 0.8 \\
    6$\times$6, 21
    & 66 & 17 & 0.25 & -651.0 & -647.7 & 0.5 \\
    \begin{tabular}{c@{}c@{}} 6$\times$11, \\ 21 \end{tabular}
    & 63 & 25 & 0.40 & -653.1 & -647.5 & 0.9 \\
    \begin{tabular}{c@{}c@{}} 11$\times$11, \\ 21 \end{tabular}
    & 54 & 41 & 0.67 & -653.1 & -648.9 & 0.6 \\
    \begin{tabular}{c@{}c@{}} 11$\times$11, \\ 41 \end{tabular}
    & 93 & 87 & 0.94 & -639.4 & -633.4 & 0.9 \\
    \begin{tabular}{c@{}c@{}} 11$\times$21, \\ 41 \end{tabular}
    & 72 & 121 & 1.7 & -649.7 & -649.1 & 0.1 \\
    \begin{tabular}{c@{}c@{}} 21$\times$21, \\ 41 \end{tabular}
    & 54 & 168 & 3.1 & -647.1 & -638.7 & 1.3 \\
    \begin{tabular}{c@{}c@{}} 21$\times$21, \\ 201 \end{tabular}
    & 42 & 442 & 10.5 & -652.3 & -638.7 & 1.3 \\
    \begin{tabular}{c@{}c@{}} 21$\times$101, \\ 201 \end{tabular}
    & 42 & 2092 & 49.8 & -654.2 & -651.5 & 0.4 \\
    \begin{tabular}{c@{}c@{}} 101$\times$101, \\ 201 \end{tabular}
    & 41 & 10781 & 263.4 & -650.4 & -643.5 & 1.0 \\
  \end{tabular}
  \captionsetup{width=\textwidth}
  \caption{
    Detailed numerical performances for implementations
    of the $\mu$SDP method, characterized by the size
    of the discrete grids for state and control variables
    in the first column.
    Other columns report the
    number of iterations performed
    (second column),
    time performances (third and fourth columns),
    together with the lower bound
    $\underline{\ValueFunction}_{0}(\state_0 \sep \param\opt)$
    (fifth column), the expected simulation cost
    $\overline{\ValueFunction}_{0}(\state_0 \sep \param\opt)$
    (sixth column),
    and the estimation gap (seventh column)
    expressed as a percentage of $\underline{\ValueFunction}_{0}(\state_0 \sep \param\opt)$.
    For columns 2-6, the lower the values the better the performance
    of the instance.}
  \label{tab:muSDP_PGD}
\end{table}

\begin{table}[H]
  \centering
  \begin{tabular}{l|rrrrrr}
    $\boldsymbol{k}$
    & \begin{tabular}{@{}c@{}}
        \textbf{Iterative} \\ \textbf{steps}\end{tabular}
    & \begin{tabular}{@{}c@{}c@{}}
        \textbf{Overall} \\ \textbf{time} \\
        (seconds)\end{tabular}
    & \begin{tabular}{@{}c@{}c@{}}
        \textbf{Avg. time /} \\ \textbf{oracle call} \\ (seconds)\end{tabular}
    & \begin{tabular}{@{}c@{}c@{}}
        $\boldsymbol{\underline{\ValueFunction}_{0}(\state_0 \sep \param\opt)}$ \\
        \textbf{in \eqref{eq:estimating_bounds}} \\ (\euro)\end{tabular}
    & \begin{tabular}{@{}c@{}c@{}}
        $\boldsymbol{\overline{\ValueFunction}_{0}(\state_0 \sep \param\opt)}$ \\
        \textbf{in \eqref{eq:estimating_bounds}} \\ (\euro)\end{tabular}
    & \begin{tabular}{@{}c@{}}
        \textbf{Gap} \\ (\%)\end{tabular}
    \\
    \hline
    10 & 7 & 11 & 1.5 & -440.0 & -438.1 & 0.4 \\
    20 & 60 & 155 & 2.6 & -575.4 & -566.7 & 1.5 \\
    40 & 100 & 504 & 5.0 & -643.5 & -642.9 & 0.1 \\
    80 & 100 & 1061 & 10.6 & -654.2 & -646.4 & 1.2 \\
    150 & 100 & 2173 & 21.7 & -658.7 & -652.9 & 0.8 \\
    250 & 34 & 1428 & 42.0 & -655.1 & -644.1 & 1.7 \\
    500 & 15 & 1622 & 108.2 & -653.4 & -648.0 & 0.8 \\
    750 & 16 & 3448 & 216.9 & -653.7 & -644.5 & 1.4 \\
    1000 & 12 & 3912 & 323.4 & -651.6 & -645.0 & 1.0 \\
  \end{tabular}
  \captionsetup{width=\textwidth}
  \caption{
    Detailed numerical performances for implementations
    of the $k$SDDP method, characterized by $k$ in the first column.
    Other columns report the
    number of iterations performed
    (second column),
    time performances (third and fourth columns),
    together with the lower bound
    $\underline{\ValueFunction}_{0}(\state_0 \sep \param\opt)$
    (fifth column), the expected simulation cost
    $\overline{\ValueFunction}_{0}(\state_0 \sep \param\opt)$
    (sixth column),
    and the estimation gap (seventh column)
    expressed as a percentage of $\underline{\ValueFunction}_{0}(\state_0 \sep \param\opt)$.
    For columns 2-6, the lower the values the better the performance
    of the instance}
  \label{tab:kSDDP_PSM}
\end{table}



\begin{thebibliography}{36}
\providecommand{\natexlab}[1]{#1}
\providecommand{\url}[1]{\texttt{#1}}
\expandafter\ifx\csname urlstyle\endcsname\relax
  \providecommand{\doi}[1]{doi: #1}\else
  \providecommand{\doi}{doi: \begingroup \urlstyle{rm}\Url}\fi

\bibitem[Ausgrid(2021)]{ausgrid}
Ausgrid.
\newblock Solar home electricity data, 2021.
\newblock URL
  \url{https://www.ausgrid.com.au/Industry/Our-Research/Data-to-share/Solar-home-electricity-data}.

\bibitem[Bauschke and Combettes(2017)]{Bauschke-Combettes:2017}
H.~H. Bauschke and P.~L. Combettes.
\newblock \emph{Convex Analysis and Monotone Operator Theory in {H}ilbert
  Spaces}.
\newblock CMS Books in Mathematics/Ouvrages de Math\'ematiques de la SMC.
  Springer-Verlag, second edition, 2017.

\bibitem[Beck(2017)]{beck2017first}
A.~Beck.
\newblock \emph{First-Order Methods in Optimization}.
\newblock SIAM, 2017.

\bibitem[Bellman(1957)]{bellman1957dp}
R.~Bellman.
\newblock \emph{Dynamic Programming}.
\newblock Princeton University Press, 1957.

\bibitem[Bergounioux et~al.(1998)Bergounioux, Haddou, Hinterm{\"u}ller, and
  Kunisch]{bergounioux1998comparison}
M.~Bergounioux, M.~Haddou, M.~Hinterm{\"u}ller, and K.~Kunisch.
\newblock A comparison of interior point methods and a {Moreau-Yosida} based
  active set strategy for constrained optimal control problems.
\newblock In \emph{SIAM Journal on Optimization}. Citeseer, 1998.

\bibitem[Bertsekas(1995)]{bertsekas1995dynamic}
D.~P. Bertsekas.
\newblock \emph{Dynamic Programming and Optimal Control}, volume~1.
\newblock Athena Scientific Belmont, MA, 1995.

\bibitem[Bezanson et~al.(2012)Bezanson, Karpinski, Shah, and
  Edelman]{bezanson2012julia}
J.~Bezanson, S.~Karpinski, V.~B. Shah, and A.~Edelman.
\newblock Julia: A fast dynamic language for technical computing.
\newblock \emph{arXiv preprint arXiv:1209.5145}, 2012.

\bibitem[Bonnans and Shapiro(2013)]{bonnans2013perturbation}
J.~F. Bonnans and A.~Shapiro.
\newblock \emph{Perturbation Analysis of Optimization Problems}.
\newblock Springer, 2013.

\bibitem[Brown and Zhang(2022)]{brown2022strength}
D.~B. Brown and J.~Zhang.
\newblock On the strength of relaxations of weakly coupled stochastic dynamic
  programs.
\newblock \emph{Operations Research}, 2022.
\newblock \doi{10.1287/opre.2022.2287}.

\bibitem[Carpentier et~al.(2015)Carpentier, Chancelier, Cohen, and {De
  Lara}]{Carpentier-Chancelier-Cohen-DeLara:2015}
P.~Carpentier, J.-P. Chancelier, G.~Cohen, and M.~{De Lara}.
\newblock \emph{Stochastic Multi-Stage Optimization. At the Crossroads between
  Discrete Time Stochastic Control and Stochastic Programming}.
\newblock Springer-Verlag, Berlin, 2015.

\bibitem[Carpentier et~al.(2020)Carpentier, Chancelier, De~Lara, and
  Pacaud]{carpentier2020mixed}
P.~Carpentier, J.-P. Chancelier, M.~De~Lara, and F.~Pacaud.
\newblock Mixed spatial and temporal decompositions for large-scale multistage
  stochastic optimization problems.
\newblock \emph{Journal of Optimization Theory and Applications}, 186\penalty0
  (3):\penalty0 985--1005, 2020.

\bibitem[Cen et~al.(2012)Cen, Bonnans, and Christel]{cen:inria-00579668}
Z.~Cen, J.~F. Bonnans, and T.~Christel.
\newblock {Sensitivity analysis of energy contracts management problem by
  stochastic programming techniques}.
\newblock In R.~Carmona, P.~D. Moral, P.~Hu, and N.~Oudjane, editors,
  \emph{{Numerical Methods in Finance}}, volume 12 (2012) of \emph{Springer
  Proceeding in Mathematics}, pages 447--471. {Springer}, 2012.

\bibitem[Cervone et~al.(2015)Cervone, Santini, Teodori, and {Donatella
  Zaccagnini Romito}]{team2015impact}
A.~Cervone, E.~Santini, S.~Teodori, and {Donatella Zaccagnini Romito}.
\newblock Impact of regulatory rules on economic performance of {PV} power
  plants.
\newblock \emph{Renewable Energy}, 74:\penalty0 78--86, 2015.

\bibitem[{Commission de régulation de l'énergie ({CRE})}(2019)]{CRE}
{Commission de régulation de l'énergie ({CRE})}.
\newblock Cahier des charges des appels d’offres portant sur la réalisation
  et l’exploitation d’installations de production d’électricité à
  partir de l’énergie solaire et situées dans les zones non
  interconnectées, 2019.

\bibitem[Dellacherie and Meyer(1975)]{Dellacherie-Meyer:1975}
C.~Dellacherie and P.~A. Meyer.
\newblock \emph{Probabilit\'es et Potentiel}.
\newblock Hermann, Paris, 1975.

\bibitem[Dowson and Kapelevich(2020)]{dowson_sddp.jl}
O.~Dowson and L.~Kapelevich.
\newblock {SDDP}.jl: a {Julia} package for stochastic dual dynamic programming.
\newblock \emph{INFORMS Journal on Computing}, 2020.
\newblock \doi{https://doi.org/10.1287/ijoc.2020.0987}.
\newblock Articles in Advance.

\bibitem[Girardeau et~al.(2015)Girardeau, Lecl{\`e}re, and
  Philpott]{girardeau2015convergence}
P.~Girardeau, V.~Lecl{\`e}re, and A.~B. Philpott.
\newblock On the convergence of decomposition methods for multistage stochastic
  convex programs.
\newblock \emph{Mathematics of Operations Research}, 40\penalty0 (1):\penalty0
  130--145, 2015.

\bibitem[Guigues(2016)]{guigues2016convergence}
V.~Guigues.
\newblock Convergence analysis of sampling-based decomposition methods for
  risk-averse multistage stochastic convex programs.
\newblock \emph{SIAM Journal on Optimization}, 26\penalty0 (4):\penalty0
  2468--2494, 2016.

\bibitem[Guigues(2020)]{guigues2020inexact}
V.~Guigues.
\newblock Inexact cuts in stochastic dual dynamic programming.
\newblock \emph{SIAM Journal on Optimization}, 30\penalty0 (1):\penalty0
  407--438, 2020.

\bibitem[Guigues et~al.(2023)Guigues, Shapiro, and Cheng]{guigues2023duality}
V.~Guigues, A.~Shapiro, and Y.~Cheng.
\newblock Duality and sensitivity analysis of multistage linear stochastic
  programs.
\newblock \emph{European Journal of Operational Research}, 308\penalty0
  (2):\penalty0 752--767, 2023.

\bibitem[Le~Franc(2021)]{le2021subdifferentiability}
A.~Le~Franc.
\newblock \emph{Subdifferentiability in convex and stochastic optimization
  applied to renewable power systems}.
\newblock PhD thesis, {\'E}cole des Ponts ParisTech, 2021.

\bibitem[Lecl{\`e}re et~al.(2020)Lecl{\`e}re, Carpentier, Chancelier, Lenoir,
  and Pacaud]{leclere2020exact}
V.~Lecl{\`e}re, P.~Carpentier, J.-P. Chancelier, A.~Lenoir, and F.~Pacaud.
\newblock Exact converging bounds for stochastic dual dynamic programming via
  {F}enchel duality.
\newblock \emph{SIAM Journal on Optimization}, 30\penalty0 (2):\penalty0
  1223--1250, 2020.

\bibitem[Mordukhovich and Nam(2013)]{Mordukhovich-2013}
B.~S. Mordukhovich and N.~M. Nam.
\newblock \emph{An Easy Path to Convex Analysis and Applications}.
\newblock Morgan \& Claypool Publishers, 1st edition, 2013.
\newblock ISBN 1627052372.

\bibitem[Moreau(1965)]{moreau1965proximite}
J.~J. Moreau.
\newblock Proximit{\'e} et dualit{\'e} dans un espace {Hilbertien}.
\newblock \emph{Bulletin de la Soci{\'e}t{\'e} math{\'e}matique de France},
  93:\penalty0 273--299, 1965.

\bibitem[N'Goran et~al.(2019)N'Goran, Daugrois, Lotteau, and
  Demassey]{n2019optimal}
A.~N'Goran, B.~Daugrois, M.~Lotteau, and S.~Demassey.
\newblock Optimal engagement and operation of a grid-connected {PV}/battery
  system.
\newblock In \emph{2019 IEEE PES Innovative Smart Grid Technologies Europe
  (ISGT-Europe)}, pages 1--5. IEEE, 2019.

\bibitem[Ortega-Guti{\'e}rrez and Cruz-Su{\'a}rez(2021)]{ortega2021moreau}
R.~I. Ortega-Guti{\'e}rrez and H.~Cruz-Su{\'a}rez.
\newblock A {Moreau-Yosida} regularization for {Markov} decision processes.
\newblock \emph{Proyecciones (Antofagasta)}, 40\penalty0 (1):\penalty0
  117--137, 2021.

\bibitem[Pereira and Pinto(1991)]{pereira1991multi}
M.~V. Pereira and L.~M. Pinto.
\newblock Multi-stage stochastic optimization applied to energy planning.
\newblock \emph{Mathematical Programming}, 52\penalty0 (1):\penalty0 359--375,
  1991.

\bibitem[Pflaum et~al.(2017)Pflaum, Alamir, and Lamoudi]{pflaum2017battery}
P.~Pflaum, M.~Alamir, and M.~Y. Lamoudi.
\newblock Battery sizing for {PV} power plants under regulations using
  randomized algorithms.
\newblock \emph{Renewable Energy}, 113:\penalty0 596--607, 2017.

\bibitem[Philpott and Guan(2008)]{philpott2008convergence}
A.~B. Philpott and Z.~Guan.
\newblock On the convergence of stochastic dual dynamic programming and related
  methods.
\newblock \emph{Operations Research Letters}, 36\penalty0 (4):\penalty0
  450--455, 2008.

\bibitem[Puterman(1994)]{puterman94}
M.~L. Puterman.
\newblock \emph{Markov Decision Processes: Discrete Stochastic Dynamic
  Programming}.
\newblock John Wiley \& Sons, Inc., 1rst edition, 1994.

\bibitem[Rockafellar and Wets(2009)]{rockafellar2009variational}
R.~T. Rockafellar and R.~J.-B. Wets.
\newblock \emph{Variational Analysis}, volume 317.
\newblock Springer, 2009.

\bibitem[Rujeerapaiboon et~al.(2018)Rujeerapaiboon, Schindler, Kuhn, and
  Wiesemann]{rujeerapaiboon2018scenario}
N.~Rujeerapaiboon, K.~Schindler, D.~Kuhn, and W.~Wiesemann.
\newblock Scenario reduction revisited: Fundamental limits and guarantees.
\newblock \emph{Mathematical Programming}, pages 1--36, 2018.

\bibitem[Shapiro(2011)]{shapiro2011analysis}
A.~Shapiro.
\newblock Analysis of stochastic dual dynamic programming method.
\newblock \emph{European Journal of Operational Research}, 209\penalty0
  (1):\penalty0 63--72, 2011.

\bibitem[Shapiro et~al.(2014)Shapiro, Dentcheva, and
  Ruszczynski]{Shapiro-Dentcheva-Ruszczynski:2014}
A.~Shapiro, D.~Dentcheva, and A.~Ruszczynski.
\newblock \emph{{Lectures on Stochastic Programming: Modeling and Theory}}.
\newblock The Society for Industrial and Applied Mathematics and the
  Mathematical Programming Society, Philadelphia, USA, second edition, 2014.

\bibitem[Ter{\c{c}}a and Wozabal(2020)]{tercca2020envelope}
G.~Ter{\c{c}}a and D.~Wozabal.
\newblock Envelope theorems for multistage linear stochastic optimization.
\newblock \emph{Operations Research}, 2020.

\bibitem[Yosida(1971)]{yosida1971functional}
K.~Yosida.
\newblock \emph{Functional Analysis}.
\newblock Springer Berlin Heidelberg, 1971.

\end{thebibliography}
\end{document}